\documentclass{article}
\usepackage{amsmath, amssymb, amsthm, relsize}

\theoremstyle{plain}
\newtheorem{theorem}{Theorem}[section]
\newtheorem{lemma}[theorem]{Lemma}

\theoremstyle{definition}

\theoremstyle{remark}
\newtheorem{remark}{Remark}

\newcommand{\G}{\Gamma}
\newcommand{\U}{\Upsilon}

\DeclareMathOperator*{\OmSum}{\mathlarger{\mathlarger{\Omega}}}
\DeclareMathOperator*{\MhSum}{\mathlarger{\mathlarger{\mho}}}

\begin{document}

\title{$\Delta y = e^{sy}$ or:\\ How I Learned to Stop Worrying and Love the $\G$-function.}

\author{James David Nixon\\
	JmsNxn92@gmail.com\\
	University of Toronto}

\maketitle

\begin{abstract}
For a nice holomorphic function $f(s, z)$ in two variables, a respective holomorphic Gamma function $\G = \G_f$ is constructed, such that $f(s, \G(s)) = \G(s + 1)$. Along the way, we fall through a rabbit hole of infinite compositions, First Order Difference Equations, and absurd functional equations... This paper is orchestrated around an investigation into the unconventional equation $\Delta y = y(s+1)-y(s) = e^{sy(s)}$ and its solutions in the complex plane.
\end{abstract}

\emph{Keywords:} Complex Analysis; Difference Equations; Gamma Function; Infinite Compositions\\

\emph{2010 Mathematics Subject Classification:} 30D05; 39A45; 39B32; 32A10\\

\section{Introduction}\label{sec1}
\setcounter{equation}{0}

Leonhard Euler's Gamma function has been around. And this nearly three hundred year old grand-mother of holomorphic functions still gets around. She pops up in the most unusual places. From the Riemann zeta function \cite{ref7}, to the volume of an $n$-ball. When Euler first wrote down the pivotal equation:

\[
y = 1 \times 2 \times 3 \times ... \times n = \int (-lx)^n\,dx
\]

in a letter to Christian Goldbach \cite{ref6, ref11} after a series of correspondence with Daniel Bernoulli \cite{ref11}, none of them could've predicted the ubiquitous future the function was about to have.\footnote{In more modern notation this becomes the familiar expression $y = n! = \int_0^1(-\log(x))^n dx$.}

It's greatly owed to the odd and strange functional equation $(n + 1)y(n) =y(n + 1)$ for $n \in \mathbb{R}^+$ that the three were interested in the soon to be `analytic factorial." This became the $\Pi$-function to fellows like Gauss \cite{ref14}; and then shifted by Legendre, on a note by Euler \cite{ref11}, to what we now call the $\G$-function \cite{ref6, ref11, ref14}. In the climate of 18th century mathematics, this functional equation was no doubt a peculiarity.

The Gamma function aged like fine cognac, she found herself spread across nearly every field of analysis--and nearly every generation of mathematicians has had something new to say about her, to paraphrase Davis \cite{ref6}. Her aesthetics seem boundless: from a nice reflection formula she shares with an even greater ancestor of mathematics sin(z) \cite{ref1, ref5, ref14}; a neat multiplication formula \cite{ref1}; it's a celebrity in closed form solutions of definite integrals; unparalleled importance in fractional calculus \cite{ref13}; a neat and tidy uniqueness criterion due to Bohr-Mollerup \cite{ref2}; and rare young generalizations, like what we like to call the Barnes G-function \cite{ref3}. And on and on and on...

Her functional equation $s\G(s) = \G(s + 1)$ is as famed a functional equation as say, $e^z \cdot e^w = e^{z+w}$. And in a modern mathematician's library of functions, the $\G$-function's rather important to say the least. This paper isn't concerned so much about why the Gamma function is important. This paper is simply assuming its importance (on good faith), and running on luck that a broad generalization may gain interest. The Gamma function, she gets around; no doubt the apples don't fall far from the tree. This paper concerns itself with why we \emph{love} the Gamma function. And--\emph{There's the rub!}

\subsection{Setting up the discussion}

The $\G$-function's functional equation $s\G(s) = \G(s + 1)$ is rather simple to express, but it wasn't back when Bernoulli was still working on solving first order ODEs. (Without Euler's $f$ at $x$ notation $fx$, try writing out the Gamma functional equation.) There's a kernel of generalization in this functional equation, if only it's written more abstractly. To set a different light on it and shed some perspective on what we're going to do in this paper: let's write $f(s, z) = sz$. Instead of writing $s\G(s) = \G(s + 1)$ let's unpack the equation to see what it's really doing,

\begin{equation}
f(s, \G(s)) = \G(s + 1)
\label{eq:1A}
\end{equation}

This may be a trivial identification so it's rather uninspiring. But, what if $f$ were a different function than the simple one $f(s, z) = sz$? What if, by chance, $f(s, z)$ were something more along the lines of:

\begin{equation}
f(s, z) = z + e^{sz}
\label{eq:1B}
\end{equation}

How about if it were:

\begin{equation}
f(s, z) = z + \frac{z^2}{s^2} + \frac{z^3}{s^3}
\label{eq:1C}
\end{equation}

Let's go out of our way:

\begin{equation}
f(s, z) = z\cos(2^sz) + \frac{\sin(z)}{s^2 +1} + \frac{\sin^2(z)}{(s^2+1)^2}
\label{eq:1D}
\end{equation}

Can there still exist a nice function $\G = \G_f$ such that $f(s, \G(s)) = \G(s + 1)$?

\subsection{A relation to Picard-Lindel\"{o}f}
This problem is paralleled by the problem Picard and Lindel\"{o}f faced when examining the expression $y' = g(s, y)$ before arriving at their eponymous theorem on solutions to First Order Differential Equations \cite{ref4}. That, in summary, given nice Lipschitz $g$, a solution to $y'= g(s, y)$ exists locally. Changing the operator from differentiation $y'$ to $\mathcal{T} y(s) = y(s + 1)$, we are in a similar situation, we're looking for solutions of $y(s+1) = \mathcal{T} y = f(s, y)$ given \emph{nice} $f$. Solving said equations involve a certain care. Especially when determining what we mean by \emph{nice}.

In this paper we will construct functions $\G_f$ satisfying the above equation $f(s, \G_f (s)) = \G_f (s + 1)$, for reasonably well-behaved $f$. Again we'll have to carve out what reasonably well behaved means, but the net is rather large. To paint a picture, the above functions in: \ref{eq:1B}, \ref{eq:1C}, \ref{eq:1D} all have a respective Gamma function; though each is wildly more chaotic than the last all are solvable through the same general mechanism. We will focus especially on the $f$ in Equation \ref{eq:1B}. We will lead by example; in tweaking the results of this paper, a mosaic of functions $f$ for which $\G_f$ can be constructed will start to appear.

For a solution $y$ to $y(s + 1) = \mathcal{T} y = f(s, y)$ and $f(s, z) = z + e^{sz}$, means $y(s+ 1) = e^{sy} +y$, or $\Delta y = y(s+ 1)-y(s) = e^{sy(s)}$. From this identification it can be deduced this problem is sistered by First Oder Difference Equations. Namely, equations of the form $\Delta y = y(s+ 1)-y(s) = q(s, y)$ for a reasonably well behaved, \emph{nice} function $q$. Where the relationship between $f$ and $q$ is simply a change of variables $z + q(s, z) = f(s, z)$.

Phrasing the problem in this language makes the comparison to the existence problem of differential equations slightly more apparent. We are swapping up differentiations $y'$ for differences $\Delta y$--this is not an easy task. In contrast to the differential case: whereas a nice Lipschitz condition on $g(s, z)$ verified the existence of a suitable $y$ such that $y' = g(s, y)$, giving nice local solutions. In the difference equation setting it is imperative we speak in global terms. Which is to say we will talk of $y = \G_f$ being defined on rather large domains rather than small neighborhoods.

The results which follow can be stated in the language of either operator: $\mathcal{T} y = y(s + 1)$ or $\Delta y = y(s + 1) - y(s)$; we will freely interchange vocabulary. We will not stress too much of either/or. But the general sustenance is clearer in the difference equation setting. The conditions on $q(s, z)$ which need to be satisfied in order for a solution $\Delta y = q(s, y)$ to exist are difficult to phrase without sufficient context. But the gist is: the following sum must converge compactly normally on its domain of holomorphy (in the sense of Remmert Reinhold [14, 15]) in both variables:

\begin{equation}
\sum_{j=1}^\infty q(s-j,z)
\label{eq:1E}
\end{equation}

This is the crux of the method. It is inspired by the Cambridge fellow Srinivasa Ramanujan's work on difference equations. Ramanujan spent his time worrying about the indefinite sum though, which reduces to the case where $q(s, z) = q(s)$ is constant in $z$. Ramanujan noticed the function:

\[
y(s) = \sum_{j=1}^\infty q(s - j)
\]

satisfied the difference equation $\Delta y = q(s)$, and devoted his time towards expanding the domain of functions in which this line of argument works. Though his results are framed as a matter of summing divergent series. Our work is of a slight generalization. We will switch it up and loosen the conditions by letting $q$ be holomorphic in $z$ but not necessarily constant. However, this weak form of Ramanujan Summation is recovered from our results if we take $q$ to be constant in $z$.\footnote{If we were to let $z + q(s, z) = f(s, z) = f(z)$ be constant in $s$, then we would arrive at an inverse Abel equation $\G(s + 1) = f(\G (s))$, which deserves a far more careful treatment than the author has space in this paper.}

It is important to note now, as the author hasn't done so yet; we are dealing entirely in the realm of complex analysis. The functions we are talking about are always holomorphic. This further distances us from the differential equation case, where Picard and Lindel\"{o}f framed their results in a purely real analysis scenario. However, vivisecting the proofs which follow it is not hard to see a manner of phrasing these results in a real analysis setting; the author simply finds the results more at home in complex analysis. Just as the $\G$-function is at home in complex analysis.

The rabbit hole we take is a little unconventional, and so therefore, the presentation and the writing style must follow suit. We must be as unconventional as Euler and Bernoulli were when they tried to interpolate the factorial \cite{ref6, ref11}, and keep its functional equation (back when `functional equation" didn't really exist in a mathematician's lexicon the way it does now).

\subsection{A detour through Infinite Compositions}

We start by reconstructing a shamefully under-represented area of mathematics. This foray of mathematics is the study of infinite compositions of holomorphic functions. There exists papers scattered across journals, and the off hand remarks on internet forums, but as remarked, there is a sad vacancy in its representation. For interesting reviews the author refers to the following papers \cite{ref8, ref9, ref10, ref12}. The work of John Gill is especially interesting. For the reader's convenience, the author avoids using any of the actual math in these papers. He thinks he's found a shorter way. 

An infinite compostion of holomorphic functions is as its name suggests. It's a sequence of holomorphic functions that are composed together sequentially. There are two ways to do this because composition is non-commutative. For $\{\phi_j\}_{j=1}^\infty$ a sequence of holomorphic functions taking open $\mathcal{G} \to \mathcal{G}$, take:

\begin{equation}
\lim_{n\to \infty} \phi_1(\phi_2(...\phi_n(z)))
\label{eq:1F}
\end{equation}

or take

\begin{equation}
\lim_{n\to \infty} \phi_n(\phi_{n-1}(...\phi_1(z)))
\label{eq:1G}
\end{equation}

We care about the first case in this paper, limits of the form in Equation \ref{eq:1F}. In plainer English, we care about $\lim_{n \to \infty} \Phi_n(z)$ where $\Phi_n(z) = \Phi_{n-1}(\phi_n(z))$ with the initial condition $\Phi_1(z) = \phi_1(z)$. These are known as Inner Compositions as we add terms on the inside rather than the outside. Naturally then, limits of the form in Equation \ref{eq:1G} are referred to as Outer Compositions. An infinite composition is something rarely found these days. Though the language betrays us, for an infinite sum and an infinite product are both kinds of infinite compositions. For instance if $\phi_j(z) = z + a_j$ for a sequence of complex
numbers $\{a_j\}_{j=1}^\infty$ then

\[
\phi_1(\phi_2(...\phi_n(z))) = \lim_{n \to \infty} \big{(}z + \sum_{j=1}^n a_j\big{)} = z + \sum_{j=1}^\infty a_j
\]

Equally so, if $\phi_j(z) = b_j z$ for a sequence of complex numbers $\{b_j\}_{j=1}^\infty$ then

\[
\lim_{n\to\infty} \phi_1(\phi_2(...\phi_n(z))) = z \prod_{j=1}^\infty b_j
\]

The simplicity of how infinite sums and infinite products are infinite compositions can be trivial. Perhaps the closest thing that can jog the idea of what an infinite composition is, using an average library of ideas, are continued fractions. Wherein, if we define a sequence of functions $\phi_j(z) = \frac{d_j}{c_j+z}$ for any sequence of complex numbers $\{d_j\}_{j=1}^\infty$, $\{c_j\}_{j=1}^\infty$ then the sequence of nested fractions:

\[
\cfrac{d_1}{c_1+\cfrac{d_2}{c_2+\cfrac{d_3}{c_3+\dots + \cfrac{d_n}{c_n+z}}}}
\]

up to $n$ can be written as the composition of $\phi_j$ up to $n$ at $z$: $\phi_1(\phi_2(...\phi_n(z)))$. Letting $n \to \infty$ gives us our continued fraction--foregoing convergence issues. But what if we played around more with $\phi_j$? What if we let it be something other than a standard function like $\frac{d_j}{c_j+z}$, $b_j z$, or $z+a_j$? Letting the sequence of functions be arbitrary holomorphic functions and letting $n\to\infty$ raises the question: where and when will it converge? To give more clarity to this, let us introduce an Euler-type summation/product notation:

\begin{equation}
\OmSum_{j=n}^m \phi_j = \phi_n(\phi_{n+1}(...\phi_m(z)))
\label{eq:1H}
\end{equation}

With this our question can be written more precisely: When is $\OmSum_{j=1}^\infty \phi_j(z)$ a function? The short answer is, at least whenever the following sum converges compactly normally for $z \in \mathcal{G}$:

\[
\sum_{j=1}^\infty \phi_j(z) -z
\]

Which is to say when $\phi_j: \mathcal{G} \to \mathcal{G}$ and for all compact disks $\mathcal{K} \subset \mathcal{G}$ the sum

\[
\sum_{j=1}^\infty \sup_{z\in \mathcal{K}} |\phi_j (z) - z| < \infty
\]

then $\OmSum_{j=1}^\infty \phi_j$ is a holomorphic function on $\mathcal{G}$.

\begin{remark}
The author will now briefly point out the notation used in Equation \ref{eq:1H} is not standardized, but the author finds it preferable to the notation typically used in the field. And since the object itself is rather novel, we are in a position to attempt to standardize this notation. It is commonplace to write $\mathcal{R}$ in place of $\Omega$, where the $\mathcal{R}$  is taken to stand for Right Handed Composition, or Right Composition. The author does not like the use of a letter from the Latin alphabet in such a circumstance and opted for a capital Greek letter as it is more in-tune with Euler's notation $\Sigma$ and $\Pi$. The letter $\Omega$ was chosen as if we invert it to $\mho$ we still arrive at a Greek letter. Therein the Outer Composition can be written:

\begin{equation}
\MhSum_{j=n}^m = \phi_m(\phi_{m-1}(...\phi_n(z)))
\label{eq:1I}
\end{equation}

And there is a nice similarity between the two. Sadly, the only way to remember which type of composition is assigned to which Greek symbol ($\Omega$ or $\mho$) is through rote. Whereas the more standard notation would use $\mathcal{L}$ in place of $\mho$, to stand for Left Handed Composition, or Left Composition--and from this $\mathcal{R}$ and $\mathcal{L}$ identification its easier to remember which composition we are talking about. The author finds this point moot though, as its comparable to new students needing to learn $\Sigma$ is a sum and $\Pi$ is a product despite the identification being nearly arbitrary.
\end{remark}

To see why a study of infinite composition helps us on the road to construct $\G = \G_f$ for arbitrary $f(s, z)$ where $f(s, \G(s)) = \G(s + 1)$ is a bit tricky. But no doubt, the yarn will unravel as we go along. And remember, 18th century mathematicians (like Euler and Bernoulli and Goldbach) needed to get around to get with the Gamma function. They flirted with: imaginary numbers, limits at infinity, the integral at infinity, the exponential function (both $e^x$ and $x^s$), the derivative, the logarithm; all before they ever got around proper. So for our purposes, we'll need to get around.

\section{Converging infinite compositions on the unit disk}\label{sec2}
\setcounter{equation}{0}

We're going to pull a rabbit out of the hat, and justify why this rabbit is the rabbit we were looking for. The author is following an expository style for this paper. For this reason, none of the results in this section will be used. These results are intended as a proof sketch of what will follow in the next and coming sections. It is helpful to think of this paper as the production of a television show; this will be the story board phase of production.

It is necessary we build up a fair amount of machinery before we can tackle First Order Difference Equations. This section is intended to familiarize the reader with the sweeping broad-stroked motions needed. Which, in honesty, is solely acclimatizing the reader with infinite compositions and its \emph{je ne sais quoi} which has made it such a difficult idea to study. For the remainder of this section the following schema is used: The functions $\{\phi_j\}_{j=1}^\infty$ are holomorphic functions taking $\mathbb{D} = \{z \in \mathbb{C}\, |\, |z| < 1\}$, the unit disk in the complex plane, to itself. We also insist that

\begin{equation}
\sum_{j=1}^\infty ||\phi_j(z) - z||_{\mathbb{D}} < \infty
\label{eq:2A}
\end{equation}

where $||b(z)||_{\mathbb{D}} = \sup_{z\in \mathbb{D}} |b(z)|$. The reader should note this is a slightly stronger condition than normal convergence on compact subsets. This is simply normal convergence of the sum $\sum_{j=1}^\infty \phi_j (z) - z$. We will start small and work our way up to normal convergence on compact subsets, which is a much more natural condition on the sequence $\phi_j$. This is our rabbit, and here's why it matters. 

We want to show that Condition \ref{eq:2A} is enough for our sequence of functions $\Phi_n = \phi_1(\phi_2(...\phi_n(z))) = \OmSum_{j=1}^n \phi_j$ converge uniformly as $n\to\infty$ on compact subsets of $\mathbb{D}$. Or to speak in a new language, that $\OmSum_{j=1}^\infty \phi_j$ is a holomorphic function.

Expressions which look like Condition \ref{eq:2A} will pop up continuously in this paper--only slightly altered each time. How they are handled will be very uniform, though. Conditions looking like Condition \ref{eq:2A} will be simply known as Summability Criterion. Their inherent value will be slowly unearthed. Let's take a simple case first and reverse engineer why Condition \ref{eq:2A} might be enough to ensure convergence of $\Phi_n = \phi_1(\phi_2(...\phi_n(z)))$ as $n\to\infty$, and is a good candidate. For the sake of the argument we will return to infinite sums and infinite products. Take the sequence of functions $h_j (z) = a_j + z$ for an arbitrary sequence $a_j \in \mathbb{C}$. Even though $h_j$ doesn't take $\mathbb{D} \to \mathbb{D}$ bear with the author momentarily. Regard:

\[
\OmSum_{j=1}^n h_j = h_1(h_2(...h_n(z))) = z + \sum_{j=1}^n a_j
\]

If this were to converge as $n \to \infty$, it would mean the sequence $\{a_j\}_{j=1}^\infty$ is summable. Well if we require

\[
\sum_{j=1}^\infty ||h_j(z) - z||_{\mathbb{D}} = \sum_{j=1}^\infty |a_j| < \infty
\]

Then, this is really just saying that $\{a_j\}_{j=1}^\infty$ is absolutely convergent. This is a bit vacuous at first glance, so let's look at products. Supposing $g_j (z) = b_j z$ is a sequence of functions for an arbitrary sequence $\{b_j\}_{j=1}^\infty$ then:

\[
\OmSum_{j=1}^n g_j = z \prod_{j=1}^n b_j
\]

If as $n \to \infty$ this converges, it merely means the infinite product of $b_j$ converges. But, in requiring

\[
\sum_{j=1}^\infty ||g_j(z) - z||_{\mathbb{D}} < \infty
\]

Then since $g_j(z) - z = b_j z - z = (b_j - 1)z$ we are given $\sum_{j=1}^\infty |b_j - 1| < \infty$ and the infinite product of $b_j$ converges absolutely. Which is something a bit more non-trivial. Naturally, if we want our condition on the convergence of infinite compositions to be as general as possible, we'll need to require that it includes infinite sums and infinite products. Condition \ref{eq:2A} is looking like a good choice so far, as it handles both cases rather nicely.

The idea to be extrapolated from Condition \ref{eq:2A}, is it implies $\phi_j$ tends to the identity function $\text{Id}(z) = z$ fast enough to be summed, and this is enough to ensure the infinite composition of $\{\phi_j\}_{j=1}^\infty$ is convergent. Just how if $a_j$ tends to zero fast enough, the sum of $\{a_j\}_{j=1}^\infty$ is convergent; and if $b_j$ tends to $1$ fast enough to be summed, the product of $\{b_j\}_{j=1}^\infty$ converges. This leads us to a theorem that indeed shows Condition \ref{eq:2A} is sufficient for $\OmSum_{j=1}^\infty \phi_j$ to be a holomorphic function taking $\mathbb{D} \to \mathbb{D}$. Many of our proofs in this paper will sister this proof, only gaining in slight complexity each time. The author asks that the reader pays close attention to the development of this proof as it is a simpler version of the main result of this paper. Despite the looseness of our language so far, the following proof will be made as precise as possible.

\begin{theorem}\label{thm2A}
If a sequence of holomorphic functions $\{\phi_j\}_{j=1}^\infty$ taking $\mathbb{D} \to \mathbb{D}$ satisfies Condition \ref{eq:2A}, namely

\[
\sum_{j=1}^\infty ||\phi_j(z) - z||_{\mathbb{D}} < \infty
\]

then the infinite composition

\[
\Phi(z) = \OmSum_{j=1}^\infty \phi_j
\]

is a holomorphic function taking $\mathbb{D}$ to itself.
\end{theorem}

\begin{proof}
Pick a point $z_0 \in \mathbb{D}$. Pick a $\delta > 0$ so that $|z - z_0| \le 2\delta$ is a compact neighborhood about the point $z_0$ within $\mathbb{D}$. We are concerned with the function $\Phi_n(z) = \phi_1(\phi_2(...\phi_n(z)))$. Taking preliminary bounds it follows

\[
||\Phi_n(z)||_{|z-z_0|\le 2 \delta} < 1
\]

which is because $\Phi_n$ lives in $\mathbb{D}$. \`A la Cauchy, about the disk $|z - z_0| < 2\delta$ we also have

\[
\frac{d^k}{dz^k} \Phi_n(z) = \Phi_n^{(k)}(z) = \frac{k!}{2\pi i}\int_{|z-z_0|=2\delta} \frac{\Phi_n(\xi)}{(\xi - z)^{k+1}}\, d\xi
\]

Now we're going to perform a small trick and shrink our domain of interest from a disk about $z_0$ with radius $2 \delta$ to a disk about $z_0$ with radius $\delta$. This allows us to write, using supremum norms across $|z - z_0| \le \delta$.

\begin{eqnarray*}
||\Phi^{(k)}_n(z)||_{|z-z_0|\le \delta} &\le& \frac{k!}{2\pi} \int_{|z-z_0|=2\delta} \frac{|\Phi_n(\xi)|}{|\xi - z|^{k+1}}\,d\xi\\
&\le& \frac{k!}{2\pi} \int_{|z-z_0| = 2\delta} \frac{1}{\delta^{k+1}}\,d\xi\\
&\le& 2 \frac{k!}{\delta^k}\\
\end{eqnarray*}

for all $k$. Where we bounded $|\frac{1}{\xi-z}| \le \frac{1}{\delta}$ when $|\xi - z_0| = 2\delta$ and $|z - z_0| \le \delta$. Remembering that $\Phi_{n+1}(z) = \Phi_n(\phi_{n+1}(z))$, and that $|\phi_{n+1}(z)-z| < ||\phi_{n+1}(z)-z||_{\mathbb{D}}$ tends to zero as $n \to \infty$ we can choose $N$ such when $n > N$ it follows $||\phi_{n+1}(z) - z||_{\mathbb{D}} < \rho < \delta$. By Taylor's theorem, we can expand $\Phi_{n+1}(z) = \Phi_n(\phi_{n+1}(z))$ about $\Phi_n$'s Taylor series about $z$:

\begin{eqnarray*}
\Phi_{n+1}(z) - \Phi_n(z) &=& \Phi_{n}(\phi_{n+1}(z)) - \Phi_n(z)\\
&=& \sum_{k=0}^\infty \Phi_n^{(k)}(z) \frac{(\phi_{n+1}(z) - z)^k}{k!} - \Phi_n(z)\\
&=& \sum_{k=1}^\infty \Phi_n^{(k)}(z) \frac{(\phi_{n+1}(z) - z)^k}{k!}\\
\end{eqnarray*}

The above series converges for at least $|z-z_0| \le \delta$ by the bounds $|\Phi^{(k)}_n(z)| \le 2 \frac{k!}{\delta^k}$ and $|\phi_{n+1}(z) - z| < \rho < \delta$. The trick to this proof is we are going to pull out one factor of $\phi_{n+1}(z) - z$ from this expression and bound the rest by a constant.

Again, using the supremum norm across $|z - z_0| \le \delta$:

\begin{eqnarray*}
||\Phi_{n+1}(z) - \Phi_{n}(z)||_{|z-z_0|\le\delta} &\le& \sum_{k=1}^\infty ||\Phi_n^{(k)}(z)||_{|z-z_0| \le \delta} \frac{||\phi_{n+1}(z) - z||_{|z-z_0|\le\delta}^k}{k!}\\
&\le& \sum_{k=1}^\infty 2 \frac{k!}{\delta^k} \frac{||\phi_{n+1}(z) - z||^k_{\mathbb{D}}}{k!}\\
&\le& 2\sum_{k=1}^\infty \frac{||\phi_{n+1}(z) - z||^k_{\mathbb{D}}}{\delta^k}\\
&\le& 2 ||\phi_{n+1}(z) - z||_{\mathbb{D}} \sum_{k=1}^\infty \frac{\rho^{k-1}}{\delta^k}\\
&\le& \frac{2}{\delta(1- \frac{\rho}{\delta})} ||\phi_{n+1}(z) - z||_{\mathbb{D}}\\
&\le& M ||\phi_{n+1}(z) - z||_{\mathbb{D}}
\end{eqnarray*}

Here the equality $M = \frac{2}{\delta(1 - \frac{\rho}{\delta})}$ was made. The proof can be concluded by noting for all $\epsilon > 0$ there is an $L$ such that when $m \ge n > L > N$

\[
\sum_{j=n}^{m-1} ||\phi_{j+1}(z) - z||_{\mathbb{D}} < \epsilon/M
\]

which is possible by our criterion of $\phi_j$ in Condition \ref{eq:2A}. Then by a simple rearrangement

\begin{eqnarray*}
||\Phi_n - \Phi_m||_{|z-z_0|\le\delta} &\le& \sum_{j=n}^{m-1} ||\Phi_{j+1} - \Phi_j ||_{|z-z_0|\le\delta}\\
&\le& \sum_{j=n}^{m-1} M||\phi_{j+1}(z) - z||_{\mathbb{D}}\\
&<& \epsilon\\
\end{eqnarray*}

Which shows about an arbitrary point $z_0 \in \mathbb{D}$ there is a compact neighborhood $|z -z_0| \le \delta$ for $\delta > 0$ in which $\phi_n$ converges uniformly. Which gives the result.
\end{proof}

This gives us our first theorem. When $\phi_j : \mathbb{D} \to \mathbb{D}$ is a sequence of holomorphic functions that converge uniformly to the identity function fast enough, then their infinite composition converges. This is nice, it should give us hope that infinite compositions may have meaning in other cases.

The real question though: Why does this matter towards our goal of solving a modified Picard-Lindel\"{o}f equation? How do we get our Gamma function $\G_f$ from this? Without going into too much detail, the reason this matters can be summed up by the following equations. Supposing the function $f(s, z)$ of our desired equation $y(s + 1) = \mathcal{T} y = f(s, y)$ is denoted $f_s(z)$, then the function:

\[
\G(s) = \OmSum_{j=1}^\infty f_{s-j}(z) = f_{s-1}(f_{s-2}(f_{s-3}(...)))
\]

can be seen to satisfy the following

\begin{eqnarray*}
f_s(\G(s)) &=& f_s(f_{s-1}(f_{s-2}(f_{s-3}(...))))\\
&=& f_{s+1-1}(f_{s+1-2}(f_{s+1-3}(...)))\\
&=&\G(s + 1)
\end{eqnarray*}

Obviously care needs to be taken when talking about holomorphy in $s$. Which is why we've built this paper in progression. I have chosen clarity of meaning over conciseness, as young auteurs are always taught. These equations form the exploit we hope to use in the coming sections. It is the exploit of the entire paper. It is the cornerstone. It's exactly how we're going to get around. We welcome you to the wonderful world of whacky infinite compositions.

\section{Infinite compositions on arbitrary domains}\label{sec3}
\setcounter{equation}{0}

Infinite compositions can make sense on the unit disk... So what does this tell us? Well it doesn't really help get the solution to $\G = \G_f$ in general instances, but it did give us a nice stretch and warm up. The reader may hopefully forgive the author, for that was its sole purpose. This section will make a stronger version of Theorem \ref{thm2A}, and consequently obsolesce everything we just did. There is a parallel between the theorem we present in this section and the one we just gave; excepting that the proceeding proof will require more nuance and care. But what exactly are we going to prove? It is natural to wonder, since infinite compositions can make sense on $\mathbb{D}$ (the unit disk), could they work on an arbitrary domain? For the rest of this section the following schema is used: The set $\mathcal{G} \subseteq \mathbb{C}$ is open. The sequence of functions $\{\phi_j\}_{j=1}^\infty$ are holomorphic and $\phi_j : \mathcal{G} \to \mathcal{G}$. And most importantly, our main condition in this section will be: if $\mathcal{K}$ is a compact disk within $\mathcal{G}$ then

\begin{equation}\label{eq:3A}
\sum_{j=1}^\infty \sup_{z \in \mathcal{K}} |\phi_j(z) - z| = \sum_{j=1}^\infty ||\phi_j(z) - z||_{\mathcal{K}} < \infty
\end{equation}

This can equivalently be rephrased as the sum $\sum_{j=1}^\infty \phi_j (z) - z$ converges compactly normally on the set $\mathcal{G}$. This terminology is home to Remmert Reinhold in \cite{ref14, ref15}. It affords us a quicker way of discussing our summability criterion. Normal convergence on compact subsets is crucial; the language was chosen as it shortens some of the proofs substantially. It is natural, as an added fact; in most circumstances when we talk about absolute convergence of a sum of functions the sum tends to actually be compactly normally convergent. Insofar as, if one happens to use the Weierstrass M-test when proving the absolute convergence of a sum, one has shown the sum is in fact compactly normally convergent. We could have simply chosen the condition that $\sum_{j=1}^\infty |\phi_j(z) - z| < \infty$ for all $z \in \mathcal{G}$ but this afforded longer, more intricate and complicated proofs to follow--and what the author felt were quite a few dead-ends.

\begin{remark} Of all the written literature the author has found on infinite compositions, Condition \ref{eq:3A} happens to be the weakest condition to guarantee convergence. He finds this condition very natural, as to talk of infinite composition, one needs only examine the convergence of a sum--and that sum is not exotic in the least. Professor John Gill comes close to a similar criterion in \cite{ref9}. Our criterion happens to imply Gill's but not the other way around. Though similar, John Gill frames the language in terms of Lipschitz conditions and is more restrictive; though it does still rely on a summability criterion. His exposition also fails to consider arbitrary sequences of functions and focuses on the case where $\phi_j (z) = z + \rho_j \psi(z)$ where $\psi$ is an arbitrary function and $\rho_j$ is a summable sequence.
\end{remark}

Condition \ref{eq:3A} is enough for the function $\OmSum_{j=1}^\infty \phi_j$ to be a holomorphic function, and take $\mathcal{G} \to \mathcal{G}$. This is a slightly altered summability criterion than the one we saw in Condition \ref{eq:2A}, it does not hold on all of $\mathcal{G}$, only on compact disks within $\mathcal{G}$.

This causes a bit of trouble that requires us to tweak our approach from Theorem \ref{thm2A}. To begin we'll need a lemma that tends to be taken for granted when talking about infinite sums and infinite products, which allows us to generalize to arbitrary  domains. This lemma wasn't needed in the first section because we had the benefit of $\sup_{z \in \mathbb{D}} |\phi_j(z) - z|$ being a summable sequence. In this section, it is comparable, excepting we have a weaker summability criterion. We do not have that $\sup_{z \in \mathcal{G}} |\phi_j (z) - z|$ is a summable sequence. We only have $\sup_{z \in \mathcal{K}} |\phi_j (z) - z|$ is a summable sequence for any compact disk $\mathcal{K}$ within $\mathcal{G}$. Since $\mathcal{G}$ could very well be all of $\mathbb{C}$ we are in a tighter bind here. This is going to cause some trouble. 

To elucidate why, since $\mathcal{G}$ need not be bounded, the sequence $\phi_n = \phi_1(\phi_2(...\phi_n(z)))$ need not be bounded on $\mathcal{G}$. In Theorem \ref{thm2A} this sequence was bounded by $1$. In this schema, the functions $\phi_j$ are bounded on $\mathcal{K}$ but do not necessarily take the compact disk $\mathcal{K}$ to itself (unlike how previously, $\phi_j : \mathbb{D} \to \mathbb{D}$ and $|\phi_j (z)| < 1$), so \`{a} priori we have no bound on the sequence of terms $\phi_1(\phi_2(...\phi_n(z)))$. Can we modify our approach so that the sequence $\Phi_n$ is bounded on $\mathcal{K}$?

What we can prove, is for all $\rho > 0$ there is an $N$ such for all $m \ge n > N$ the composition $\OmSum_{j=n}^m \phi_j$ is at most $\rho$ distance from $z$. This tells us we can find a larger compact disk about a neighborhood $|z - z_0| < \delta$, call it $\mathcal{K}$ in which $\OmSum_{j=n}^m \phi_j \Big{|}_{|z-z_0|<\delta} \subset \mathcal{K}$ for all $m \ge n > N$. This allows us to ground ourselves and unilaterally bound $||\OmSum_{j=N}^m \phi_j ||_{|z-z_0|\le\delta}$ by a bound on $\mathcal{K}$. In the last section we got this for free. We bounded $||\OmSum_{j=1}^n \phi_j|| < 1$ like it was nothing. In this case $\phi_j : \mathcal{G} \to \mathcal{G}$ and no such bound on $\mathcal{G}$ necessarily exists (encore, considering $\mathcal{G}$ may even be all of $\mathbb{C}$), so such an approach will not work. We'll have to be more clever. 

To add even more intuition to the following lemma we can again look at sums and products. If $\sum_{j=1}^\infty a_j$ converges, then for all $\rho > 0$ there is an $N$ such when $m \ge n > N$ it follows $|\sum_{j=n}^m a_j - 0| <\rho$; which is just the definition of summability. Similarly for sequences $b_j$, if $\prod_{j=1}^\infty b_j$ converges absolutely, then there is an $N$ such when $m \ge n > N$ it follows $|\prod_{j=n}^m b_j - 1| < \rho$. Both $0$ and $1$ are neutral identity elements in these scenarios. Any number added by $0$ is itself. Any number multiplied by $1$ is itself. In the compositional case, any function composed with $\text{Id}(z) = z$ is itself. It should seem natural then, for us to want $|\phi_n(\phi_{n+1}(...\phi_m(z)))-z| = |\OmSum_{j=n}^m \phi_j -z| < \rho$ for large enough $N$ with $m \ge n > N$ on an arbitrary compact disk $\mathcal{K} \subset \mathcal{G}$. Showing Condition 3.1 allows us to say this involves a careful argument.

\begin{lemma}\label{lma3A}
 If $\{\phi_j\}_{j=1}^\infty$ is a sequence of holomorphic functions taking an open set $\mathcal{G}\subseteq \mathbb{C}$ to itself and if the sequence satisfies Condition \ref{eq:3A}, namely

\[
\sum_{j=1}^\infty ||\phi_j(z) - z||_{\mathcal{K}} <\infty
\]

for any compact disk $\mathcal{K}$ within $\mathcal{G}$, then for all $\rho > 0$ there exists $N$ such when $m \ge n > N$

\[
||\OmSum_{j=n}^m \phi_j - z||_{\mathcal{K}} < \rho
\]

\end{lemma}

\begin{remark}
The following is an odd proof by induction. The author could not find an alternate proof despite attempts to simplify the reasoning. It is a rare thing to see proofs by induction in analysis done in the following manner, but the author reminds the reader that the strangeness of the object requires a strangeness in the approach. He attempted to maintain as much clarity as possible while not making the proof overly lengthy. The complexity of this proof resides most apparently in the discussion of domains, made particular in the induction process.
\end{remark}

\begin{proof}
Choose a $z_0 \in \mathcal{G}$. Choose $\mathcal{K}$ a compact disk $|z - z_0| \le P$ for some $P > 0$ within $\mathcal{G}$. Choose a $\delta > 0$ such that the points $|z - z_0| < \delta < P$ live in the interior of $\mathcal{K}$. Let $\rho_j = ||\phi_j(z) - z||_{\mathcal{K}}$. By Condition \ref{eq:3A} it is given

\[
\sum_{j=1}^\infty \rho_j < \infty
\]

We are going to take the tail of this series and manipulate it in a significant way. Choose $N$ so that $\sum_{j=N+1}^\infty \rho_j < \rho$ is as small as desired so the points $|z - z_0| < \delta + \rho < P$ still reside in the interior of $\mathcal{K}$. We denote $F_{nm} = \OmSum_{j=n}^m \phi_j = \phi_n(\phi_{n+1}(...\phi_m(z)))$. We want to show that for all $|z -z_0| < \delta$ there exists $N$ such when $m \ge n > N$ the inequality $|F_{nm}(z)-z| < \rho$ is satisfied. This in conjunction, shows for $N$ large enough if $m \ge n > N$ then $F_{nm}(z)$ is in a small enough disk about $z_0$ which still resides within $\mathcal{K}$; because $|F_{nm}(z) - z_0| < |z - z_0| + |F_{nm}(z) - z| < \delta + \rho < P$. The proof goes by using strong induction on the difference $m -n = k$; starting with the base-case $m = n$ (or $k = 0$). Recalling $F_{nn}(z) = \phi_n(z)$, this follows forwardly:

\[
|F_{nn}(z) - z| = |\phi_n(z) - z| \le ||\phi_n(z) - z||_{\mathcal{K}} = \rho_n \le \sum_{j=N+1}^\infty \rho_j < \rho
\]

Just as well by the triangle inequality $|F_{nn}(z) - z_0| \le |F_{nn}(z) - z| + |z - z_0| < \rho + \delta < P$. So $F_{nn}(z) = \phi_n(z)$ is in $\mathcal{K}$ while $|z-z_0| < \delta$ and $n > N$. For the case $m = n+1$ (when $k = 1$), the result follows, remembering the functional equation $F_{nm}(z) = \phi_n(F_{(n+1)m}(z))$. We are going to use a telescoping idea that should shed light on how the general induction process will go. We want to show $|F_{n(n+1)}(z) - z| < \rho$ using the above.

\begin{eqnarray*}
|F_{n(n+1)}(z) - z| &\le& |F_{n(n+1)}(z) - F_{(n+1)(n+1)}(z)| + |F_{(n+1)(n+1)}(z) - z|\\
&=& |\phi_n(F_{(n+1)(n+1)}) - F_{(n+1)(n+1)}| + |\phi_{n+1}(z) - z|\\
&\le& ||\phi_n(z) - z||_{\mathcal{K}} + ||\phi_{n+1}(z) - z||_{\mathcal{K}}\\
&=& \rho_n + \rho_{n+1}\\
&\le& \sum_{j=N+1}^\infty \rho_j < \rho\\
\end{eqnarray*}

In the above, we legally bounded $|\phi_n(F_{(n+1)(n+1)}) - F_{(n+1)(n+1)}|$ by $\rho_n = ||\phi_n(z) - z||_{\mathcal{K}}$ because $F_{jj} \in \mathcal{K}$ for all $j > N$ when $|z - z_0| < \delta$, which was just shown. After this, we bounded $\rho_n + \rho_{n+1}$ by $\rho > \sum_{j=N+1}^\infty \rho_j$. Similarly now, $|F_{n(n+1)}(z) - z_0| < \rho + \delta < P$ by the triangle inequality and therefore $F_{n(n+1)}$ still resides in the compact set $\mathcal{K}$. That's is the trick of this proof by induction. So when $m - n = 1$, the function $F_{nm}(z)$ lives inside $\mathcal{K}$ when $|z - z_0| < \delta$ and satisfies $|F_{nm}(z) - z| < \rho$ when $|z - z_0| < \delta.$

Continuing in this fashion, we will apply strong induction. Assume for all $m \ge n > N$ and $m - n = j < k$ that $|F_{nm}(z) - z| = |F_{n(n+j)} - z| < \rho$. Which vacuously shows, when $m-n < k$ and $m \ge n > N$ that $|F_{nm}(z)-z_0| < |F_{nm}(z)-z| + |z - z_0| < \rho + \delta < P$ and $F_{nm} = F_{n(n+j)} \in \mathcal{K}$. The inductive step is to show this implies when $m - n = k$ that $|F_{nm}(z) - z| = |F_{n(n+k)}(z) - z| < \rho$ and $F_{nm} = F_{n(n+k)} \in  \mathcal{K}$ while $m \ge n > N$ and $|z - z_0| < \delta$. The result will follow by strong induction. Well then, to do such, do what was done above but with more terms. We are going to telescope again. Again, a key thing to remember is the identity $\phi_n(F_{(n+1)m}(z)) = F_{nm}(z)$, we are going to abuse it.

\begin{eqnarray*}
|F_{n(n+k)}(z) - z| &\le& \sum_{j=0}^{k-1} |F_{(n+j)(n+k)} - F_{(n+j+1)(n+k)}| + |F_{(n+k)(n+k)} - z|\\
&=& \sum_{j=0}^{k-1} |\phi_{n+j}(F_{(n+j+1)(n+k)}) - F_{(n+j+1)(n+k)}| + |\phi_{n+k}(z) - z|\\
&<& \sum_{j=0}^k \rho_{n+j}\\
&\le& \sum_{j=N+1}^\infty \rho_j < \rho\\
\end{eqnarray*}

The term $|\phi_{n+j} (F_{(n+j+1)(n+k)})-F_{(n+j+1)(n+k)}|$ was bounded by $\rho_{n+j} = ||\phi_{n+j} (z)- z||_{\mathcal{K}}$ because $F_{(n+j+1)(n+k)} \in \mathcal{K}$. This is covered by the induction hypothesis because $(n + k) - (n + j + 1) = k - j - 1 < k$. Furthering this reasoning $|F_{nm}(z) - z_0| = |F_{n(n+k)}(z) - z_0| < \rho + \delta < P$ and so $F_{nm} = F_{n(n+k)} \in \mathcal{K}$. By strong induction therefore $|F_{nm}(z) - z| < \rho$ for all $m \ge n > N$ when $|z - z_0| < \delta$. About every point $z_0 \in \mathcal{G}$ there is a disk $|z - z_0| < \delta$ where $|F_{nm}(z) - z| < \rho$ for large enough $N$. Since every compact disk $\mathcal{K} \subset \mathcal{G}$ can be covered by a finite number of disks within $\mathcal{G}$, for all $\rho > 0$ there is an $N'$ such when $m \ge n > N'$ it follows $|F_{nm}(z) - z| < \rho$ on $\mathcal{K}$.
\end{proof}

This lemma is the stepping stone towards showing infinite compositions on arbitrary domains do make sense, and can converge. The following theorem will follow exactly as its sister Theorem \ref{thm2A} did in Section \ref{sec2}; only we'll be talking about
the tail of the composition.

\begin{remark}
It's interesting to note the following theorem provides an alternate proof of the convergence of infinite products. Namely, that if $\sum_{j=1}^\infty |b_j - 1| < \infty$ then $\prod_{j=1}^\infty b_j$ converges. The theorem below is a bit of an overkill in such a circumstance, but does work. This is interesting to point out because the proof below never uses the inequality $|\log(1 + z)| < |z|$ which is rampant when one typically proves this result.
\end{remark}

This proof is almost exactly the proof of Theorem \ref{thm2A}. We have only added a dash of complexity. The heart of it remains the same though. If this paper were a TV show, this would be the writing phase of production.

\begin{theorem}\label{thm3A}
 If $\{\phi_j\}_{j=1}^\infty$ is a sequence of holomorphic functions taking an open set $\mathcal{G} \subseteq \mathbb{C}$ to itself, and if the sequence satisfies Condition \ref{eq:3A}, as in

\[
\sum_{j=1}^\infty ||\phi_j(z) - z||_{\mathcal{K}} < \infty
\]

for any compact disk $\mathcal{K} \subset \mathcal{G}$, then the infinite composition 

\[
\Phi(z) = \OmSum_{j=1}^\infty \phi_j
\]

is a holomorphic function taking $\mathcal{G}$ to itself.
\end{theorem}

\begin{proof}
Choose a compact disk $\mathcal{K} \subset \mathcal{G}$ about an arbitrary point $z_0 \in \mathcal{G}$ with radius $P$, and choose a $\delta > 0$ such that $|z - z_0| < 2\delta < P$, so that such $z$ still reside in $\mathcal{K}$. Choose $\rho$ small enough such that $|z - z_0| < 2\delta + \rho < P$, and $z$ still resides in the interior of $\mathcal{K}$. Denote $F_{nm}(z) = \OmSum_{j=n}^m \phi_j = \phi_n(\phi_{n+1}(...\phi_m(z)))$. By Lemma \ref{lma3A}, $N$ can be chosen large enough so when $m \ge n \ge N$, $|F_{nm}(z)-z| < \rho$ for all $|z - z_0| < 2\delta$. By the triangle inequality:

\[
|F_{nm}(z) - z_0| < |z - z_0| + |F_{nm}(z) - z| < 2\delta + \rho < P
\]

for all $|z - z_0| < 2\delta$. Because of this when $m \ge n \ge N$ 

\[
||F_{nm}(z)||_{|z-z_0|\le 2\delta} \le P + |z_0| = A
\]

Taking the contour integral representation of $\frac{d^k}{dz^k}F_{nm}(z) = F^{(k)}_{nm}(z)$ we get the
expression

\[
F^{(k)}_{nm}(z) = \frac{k!}{2\pi i} \int_{|z-z_0|=2\delta} \frac{F_{nm}(\xi)}{(\xi - z)^{k+1}}\,d\xi
\]

So when we shrink the neighborhood from $|z - z_0| \le 2\delta$ to $|z - z_0| \le \delta$ and taking the supremum norm of the above expression it follows,

\begin{eqnarray*}
||F^{(k)}_{nm}(z)||_{|z-z_0|\le\delta} &\le& \frac{k!}{2\pi} \int_{|z-z_0|=2\delta} \frac{F_{nm}(\xi)}{||\xi - z||_{|z-z_0|\le\delta}^{k+1}}\,d\xi\\
&\le&\frac{k!}{2\pi} \int_{|z-z_0|=2\delta} \frac{A}{\delta^{k+1}}\,d\xi\\
&\le& 2 A \frac{k!}{\delta^k}\\
\end{eqnarray*}

Implicitly we took $\frac{1}{|\xi - z|} \le \frac{1}{\delta}$ when when $|\xi - z_0| = 2\delta$ and $|z - z_0| \le \delta$. From this point out in the proof it is possible to proceed as we did in Theorem \ref{thm2A}. It is important to note that $\rho$ chosen at the beginning of the proof should now be chosen so that $\rho < \delta$. Choose N large enough so when $m \ge N$, $\rho_{m+1} = ||\phi_{m+1}(z) - z||_{\mathcal{K}} < \rho < \delta$.\footnote{We are kind of overriding the value $\rho$ here, but it does save space.}

We use Taylor's theorem. Expanding $F_{Nm}(\phi_{m+1}(z))$ about $F_{Nm}$'s Taylor series around $z$.

\begin{eqnarray*}
F_{N(m+1)}(z) - F_{Nm}(z) &=& F_{Nm}(\phi_{m+1}(z)) - F_{Nm}(z)\\
&=& \sum_{k=1}^\infty F^{(k)}_{Nm}(z) \frac{(\phi_{m+1}(z) - z)^{k}}{k!}
\end{eqnarray*}

This series has a domain of convergence of at least $|z - z_0| \le \delta$ by the bounds $|F^{(k)}_{Nm}(z)| \le 2 A \frac{k!}{\delta^k}$ and $|\phi_{m+1}(z) - z| < \rho < \delta$. Now to do the same trick as in Theorem \ref{thm2A}. We are going to pull out one factor of $\phi_{m+1}(z) - z$ from this expression and bound the rest by a constant. Let's begin:

\begin{eqnarray*}
||F_{N(m+1)}(z) - F_{Nm}(z)||_{|z-z_0|\le\delta} &\le& \sum_{k=1}^\infty ||F^{(k)}_{Nm}(z)||_{|z-z_0|\le\delta}  \frac{||\phi_{m+1}(z) - z||_{|z-z_0|\le\delta}^{k}}{k!}\\
&\le&\sum_{k=1}^\infty 2 A \frac{k!}{\delta^k}  \frac{||\phi_{m+1}(z) - z||_{\mathcal{K}}^{k}}{k!}\\
&\le& 2 A\sum_{k=1}^\infty \frac{||\phi_{m+1}(z) - z||_{\mathcal{K}}^{k}}{\delta^k}\\
&\le&2 A ||\phi_{m+1}(z) - z||_{\mathcal{K}} \sum_{k=1}^\infty \frac{\rho^{k-1}}{\delta^k}\\
&\le& \frac{2A}{\delta(1-\frac{\rho}{\delta})}||\phi_{m+1}(z) - z||_{\mathcal{K}}\\
&=& M||\phi_{m+1}(z) - z||_{\mathcal{K}}\\
\end{eqnarray*}

Where we let $M =\frac{2A}{\delta(1-\frac{\rho}{\delta})}$. Now, by Condition \ref{eq:3A}, and the hypothesis of the theorem: for all $\epsilon > 0$ there exists an $L > N$ such that when $m \ge l > L$ we have

\[
\sum_{j=l}^{m-1} ||\phi_j(z) - z||_{\mathcal{K}} < \epsilon / M\\
\]

and therefore

\begin{eqnarray*}
||F_{Nm}(z) - F_{Nl}(z)||_{|z-z_0|\le\delta} &\le& \sum_{j=l}^{m-1} ||F_{Nj} - F_{N(j+1)}||_{|z-z_0|\le\delta}\\
&\le& \sum_{j=l}^{m-1} M||\phi_j(z) - z||_{\mathcal{K}}\\
&<& \epsilon\\
\end{eqnarray*}

Therefore $F_{Nm}$ converges uniformly on $|z - z_0| \le \delta$ as $m \to \infty$. Let $\Phi_m(z) = \phi_1(\phi_2(...\phi_{N-1}(F_{Nm}(z))))$ then $\Phi_m$ converges uniformly by the uniform continuity of the holomorphic function $\phi_1(\phi_2(...\phi_{N-1}(z)))$ on $\mathcal{K}$. Which follows because continuous functions take compact sets to compact sets and a holomorphic function restricted to a compact set is uniformly continuous there. Thus, about every point $z_0$ in $\mathcal{G}$ there is a compact disk $|z - z_0| \le \delta$ for $\delta > 0$ where $\Phi_m$ converges uniformly. Therefore $\Phi(z) = \lim_{m\to\infty}\Phi_m(z)$ is a holomorphic function taking $\mathcal{G} \to \mathcal{G}$.
\end{proof}

So we can say it clearly now with little stigma. If $\phi_j(z)$ is a sequence of holomorphic functions that tend to the identity function fast enough to be summed, then their infinite composition makes sense. Or that, in no trivial a manner, there are a broad class of instances where the function

\[
\OmSum_{j=1}^\infty \phi_j
\]

makes perfect sense. And not only makes sense, but is holomorphic. Pray, what should we do with this knowledge first?

\section{A solution to a Picard-Lindel\"{o}f problem}\label{sec4}
\setcounter{equation}{0}

\begin{remark} 
Abandon all hope ye who enter here.
\end{remark}

Proposition: There exists an entire function $\U(s)$ that solves the First Order Difference Equation $\Delta y = y(s+1)-y(s) = e^{sy(s)}$. Claim: Not only does it exist, it can be constructed. Corollary: There are many First Order Difference Equations where this is possible.

When the author was asked why this $\U$ function mattered, the author thought: `The only reason a mathematician would try to find this function is because they could, and the only reason it matters is because no one else could.' So it was with an act of absurdism and egotism that this function was invented. But it is just such a unique function that constructing it should be fun and enlightening. 

First things first, let's pull a rabbit out of the hat. Let's write the equation for the function down first, and explain why it's the right one. Let's express our $\U$ which satisfies $\U(s + 1) - \U(s) = e^{s\U(s)}$ and then let's prove it's a meaningful thing.

The trouble is though, with our current notation there is no `nice' way to write down the equation. In circumstances like this, I like to turn to Euler, who would make up notation as he was going. We're also going to take a note from Leibniz with this one. When writing the integral Leibniz made a point of closing the end of the integral with a $dx$ to mark which variable he was integrating across. This is very handy for when we're integrating with multiple variables, or when the definite integral itself depends on a variable not being integrated. I.e: $\int_0^\infty e^{-sy}\,dy$ depends on $s$, $\int_0^\infty e^{-sy}\,ds$ depends on $y$, and $\int_0^\infty e^{-sy}$ is just unclear.

We'll need something similar for infinite compositions. However, Leibniz had a nice philosophy and intuition for his exact use of $dx$; the author really doesn't have a philosophy attached to his infinite compositions. He just likes the idea of adding a term at the end to make it clear which variable we are infinitely composing across; like closing the end of a bracket. To give justification to this new notation, and why we'll need it: consider a sequence of two variable functions $\phi_j(s, z)$ and write

\[
\OmSum_{j=1}^n \phi_j(s,z)
\]

Does this mean we compose across $z$?

\[
\OmSum_{j=1}^n \phi_j(s,z) = \phi_1(s, \phi_2(s, \phi_3(s, ...\phi_n(s, z))))
\]

Or does it mean we compose across $s$?

\[
\OmSum_{j=1}^n \phi_j(s,z) = \phi_1(\phi_2(\phi_3(...\phi_n(s, z), z), z), z)
\]

Our notation is apparently unclear in such a circumstance. To this end we're going to add a little pencil-in to tell us which variable we're composing across. The author likes to use the following schema:

\begin{equation}
\OmSum_{j=1}^n \phi_j(s,z) \bullet z = \phi_1(s, \phi_2(s, \phi_3(s, ...\phi_n(s, z))))
\label{eq:4A}
\end{equation}

\begin{equation}
\OmSum_{j=1}^n \phi_j(s,z) \bullet s = \phi_1(\phi_2(\phi_3(...\phi_n(s, z), z), z), z)
\label{eq:4B}
\end{equation}

Wherein, the variable which appears after the bullet-point is the variable that is being composed across. This solves our qualm--as patchwork as it is. But with that cleared up, the equation for $\U(s)$ is elementary to write out. Making sense of it may be a hopeless affair at first, but it does get the job done well enough. Assuming $\Re(z) > 0$ and $s \in \mathbb{C}$,

\begin{equation}
\U(s,z) = \OmSum_{j=1}^\infty z + e^{(s-j)z} \bullet z
\label{eq:4C}
\end{equation}

This is going to be a nasty function to deal with, as the reader can no doubt guess. For our convenience as we progress, we're going to let $\phi_j(s,z) = z +e^{(s-j)z}$. This will ease our transition towards the general case, though only in spirit. If the reader has not absorbed our notation by this point, it will be difficult to visualize the function $\U$. The author will try one more time to anchor the absurdity of this expression in a traditional form, and make our notation seem more concrete. If we define the sequence of composite functions $\U_n = \phi_1(s, \phi_2(s, ...\phi_n(s, z)))$, where again $\phi_j(s,z) = z +e^{(s-j)z}$, then the first terms of the sequence can be writ as follows:

\begin{eqnarray*}
\U_1(s, z) &=& z + e^{(s-1)z}\\
\U_2(s, z) &=& z + e^{(s-2)z} + e^{(s-1)(z+e^{(s-2)z})}\\
\U_3(s, z) &=& z + e^{(s-3)z} + e^{(s-2)(z+e^{(s-3)z})} + e^{(s-1)(z+e^{(s-3)z}+e^{(s-2)(z+e^{(s-3)z)}})}\\
&\vdots&\\
\end{eqnarray*}

As we can see this reduces to chaos as we continue. You may begin to see why we like the clean notation of Equation \ref{eq:4C}. Also in this form it is not apparent in the slightest that the limit of these functions will be the solution of our difference equation. Considering our function in the beginning of the discussion was $f(s, z) = z+e^{sz}$, and $\phi_j(s,z) = f(s - j, z)$, we can see why $\U$ might be our desired function. Rewriting the sequence of terms for $\U_n$:

\[
\U_n(s, z) = \OmSum_{j=1}^n f(s-j, z) \bullet z = f(s - 1, f(s - 2, ...f(s - n, z)))
\]

And so

\begin{eqnarray*}
\U_n(s, z) + e^{s\U_n(s,z)} &=& f(s, \U_n(s, z))\\
&=& f(s, f(s - 1, f(s - 2, ...f(s - n, z))))\\
&=& f(s + 1 - 1, f(s + 1 - 2, ...f(s + 1 - (n + 1), z)))\\
&=& \U_{n+1}(s + 1, z)\\
\end{eqnarray*}

We can make an educated guess that the limit as $n \to \infty$ should be a solution to our modified Picard-Lindel\"{o}f problem. To digress momentarily here, the $z$-value in Equation \ref{eq:4C} serves as a parameter. It can be fixed to any value $z \in \mathbb{C}_{\Re(z) > 0}$. For each $\Re(z) > 0$ there is a different $\U$, wherein it solves our first order difference equation $\Delta y = e^{sy}$, or $\U(s+ 1)-\U(s) = e^{s\U(s)}$.

\begin{remark} 
Although not dealt with in this paper, the $z$-parameter determines the initial condition of the equation $\Delta y = e^{sy}$; our value at $y(0) = \omega$. There does exist a closed form expression for $z$ dependent on $\omega = \U(0, z)$, but this involves the other kind of infinite composition mentioned in the introduction. If you think this is messy, the initial condition problem makes the above look elegant. The author purposefully left it out as he found it spurious and unnecessary.
\end{remark}

For the first pitch when trying to show our expression for $\U$ in Equation \ref{eq:4C} converges we get something slow and steady right over the plate. When $\mathcal{K}$ is a compact disk within the right half plane $\mathbb{C}_{\Re(z)>0}$, it follows

\[
\sum_{j=1}^\infty ||\phi_j(s,z) - z||_{\mathcal{K}} < \infty
\]

Which should give us hope that $\phi_j (s, z)$ can be infinitely composed for each $s$ fixed since the sum converges geometrically and our summability criterion is met. It looks familiar. For the second pitch, we get a knuckle ball that goes which way and that. Sadly, $\phi_j (s, \cdot) : \mathbb{C}_{\Re(z)>0} \to \mathbb{C}_{\Re(z)>0}$. Or to literate: for $s$ fixed $\phi_j (s, \cdot)$ does not take the right half plane to itself. Theorem \ref{thm3A} gives us nothing then. To reiterate: in Theorem \ref{thm3A} we proved that for a sequence of functions $\{ \phi_j \}_{j=1}^\infty$ which take $\mathcal{G} \to \mathcal{G}$ and which satisfy the summability criterion of Condition \ref{eq:3A} then the expression $\OmSum_{j=1}^\infty \phi_j$ was holomorphic. Here we have a summability criterion but our sequence of functions do not send $\mathbb{C}_{\Re(z)>0} \to \mathbb{C}_{\Re(z)>0}$. There is no ideal $\mathcal{G}$ like there was in Theorem \ref{thm3A}. We are going to have to start from scratch all over again.

So, how to salvage this? The general idea is that $\phi_j (s, \cdot) : \mathbb{C} \to \mathbb{C}$ so that $\phi_1(s, \phi_2(s, ...\phi_n(s, z)))$ always exists and is a meaningful thing, but the summability criterion only holds on $\mathbb{C}_{\Re(z)>0}$. This will be enough for $\OmSum_{j=1}^n \phi_j (s, z) \bullet z = \phi_1(s, \phi_2(s, ...\phi_n(s, z)))$ to converge as $n \to \infty$ on $\mathbb{C}_{\Re(z)>0}$, however it will not live in $\mathbb{C}_{\Re(z)>0}$; it will live in $\mathbb{C}$.

But how do we let $s$ enter the picture?\\

\emph{[Enter the third summability criterion.]}\\

What is dire and important to the proceeding is the following fact. There is a cover $\{\mathcal{S}_l\}_{l=1}^\infty$ of $\mathbb{C}$ wherein for each $\mathcal{S} \in \{\mathcal{S}_l\}_{l=1}^\infty$ and each compact disk $\mathcal{K} \subset \mathbb{C}_{\Re(z)>0}$:

\begin{equation}
\sum_{j=1}^\infty ||\phi_j(s,z) - z||_{z \in \mathcal{K},s \in \mathcal{S}} < \infty
\label{eq:4D}
\end{equation}

\begin{remark} 
We choose the language of covers as it affords us an easier transition when talking of the general case--and eases the proofs of some later results on $\U$. Frankly, for the rest of this section we could simply take $\{\mathcal{S}_l\}_{l=1}^\infty$ to be the sequence of compact sets $\mathcal{S}_l = \{s \in \mathbb{C} \,|\, |s| \le l\}$, and we could abandon this language altogether and just talk about uniform convergence on compact subsets of $\mathbb{C}$. However we would like to let $\infty$ enter the picture. This is to say: our summability criterion is met if for instance one of the $\mathcal{S}$ happened to equal $\{s \in \mathbb{C}\, |\, \Re(s) < 0, |\Im(s)|< 1\}$. Allowing covers of $\mathbb{C}$ will help us talk about asymptotics of $\U(s)$ in Section \ref{sec6}.
\end{remark}

This is the condition we've been looking for. It gives way broad generalizations and soars us far from where we started in the last two sections. We are going to imitate everything we just did in the last two sections--but here is where we take score.

We are going to start with a lemma, that lemma being the parallel of Lemma \ref{lma3A}. This lemma is important for all the previous reasons Lemma \ref{lma3A} was important in Section \ref{sec3}, except we have an extra variable dangling in there. The reader should freely see that the proof goes along exactly as the proof of Lemma \ref{lma3A} went. We are quite literally going to do exactly the same thing as the last two sections excepting we're going to have to consider a variable $s$. We will also have to be more careful when talking about domains. But the heart of the proofs will be no different than Theorem \ref{thm2A}, Lemma \ref{lma3A}, Theorem \ref{lma3A}.

We're going to lose generality in the coming lemma but the reader may freely see how to apply the reasoning to more general circumstances. The author will restate this lemma in the following section. There he will put it in a general light, but by that point we'll leave the proof unsaid, as the reader should be able to fill in the -----.

So without any more time outs and due process, let us begin:

\begin{lemma}\label{lma4A}
Let $\phi_j (s, z) = z + e^{(s-j)z}$. Suppose $\mathcal{S}$ is a subset of $\mathbb{C}$ such for all compact disks $\mathcal{K}$ within $\mathbb{C}_{\Re(z)>0}$

\[
\sum_{j=1}^\infty ||\phi_j(s,z) - z||_{\mathcal{K},\mathcal{S}} < \infty
\]

For all $\rho > 0$ there is an $N$ such when $m \ge n > N$

\[
||\OmSum_{j=n}^m\phi_j(s,z) \bullet z - z||_{\mathcal{K},\mathcal{S}} < \rho
\]
\end{lemma}

\begin{remark}
Again, the set $\mathcal{S}$ need not be compact. It need only be a set in which the summability criterion is met. For instance, if $\mathcal{S} \subset \mathbb{C}$ were the set defined by the relations $|\Im(s)| < 1$ and $\Re(s)< 0$, then the summability criterion is met and the results of the lemma still follow. Therein, this lemma applies for unbounded sets in $s$ so long as the sum in Condition \ref{eq:4D} is convergent for all compact disks $\mathcal{K} \subset \mathbb{C}_{\Re(z)>0}$. This will become very important in the coming sections when discussing the behaviour of $\U$ in $s$.
\end{remark}

\begin{proof}
Choose a $z_0 \in \mathbb{C}_{\Re(z)>0}$ and a compact disk $\mathcal{K} \subset \mathbb{C}_{\Re(z)>0}$ such that its elements satisfy $|z-z_0| \le P$ where $P > 0$. Choose $\delta > 0$ such that $|z-z_0| < \delta < P$ still resides within the interior of the set $\mathcal{K}$. Set $||\phi_j (s, z) - z||_{z\in\mathcal{K},s\in\mathcal{S}} = \rho_j$ and choose $N$ large enough so that

\[
\sum_{j=N+1}^\infty \rho_j < \rho
\]

where $|z -z_0| < \delta +\rho < P$ is still in the interior of $\mathcal{K}$. This is justified, because $N$ can be chosen large enough so that $\rho$ can be made as small as possible. Denote: $F_{nm}(s, z) =\OmSum_{j=n}^m \phi_j (s, z) \bullet z = \phi_n(s, \phi_{n+1}(s, ...\phi_m(s, z)))$. The goal is to show when $m \ge n > N$, $||F_{nm}(s, z) - z||_{|z-z_0|\le\delta,s \in \mathcal{S}} < \rho$ and consequently by the triangle inequality 

\begin{eqnarray*}
||F_{nm}(s, z) - z_0||_{|z-z_0|\le\delta,s \in \mathcal{S}} &<& |z - z_0| + ||F_{nm}(s, z) - z||_{|z-z_0|\le\delta,s \in \mathcal{S}}\\
&<& \delta + \rho < P
\end{eqnarray*}

so that $F_{nm}(s, z) \in \mathcal{K}$. The natural path is to take the exact same path we took in Lemma \ref{lma3A}. We perform strong induction on the difference $m - n = k$. When $k = 0$ then $F_{nm} = F_{nn} = \phi_n(s, z)$ and

\begin{eqnarray*}
||F_{nn}(s, z) - z||_{|z-z_0|\le\delta,s \in \mathcal{S}} &\le& ||\phi_n(s, z) - z||_{\mathcal{K},\mathcal{S}}\\
&=& \rho_n \le \sum_{j=N+1}^\infty \rho_j < \rho\\
\end{eqnarray*}

By the triangle inequality this tells us $|F_{nn}(s, z)-z_0| < \rho+\delta < P$, which implies $F_{nn}(s, z) \in \mathcal{K}$ for $s \in \mathcal{S}$ and $|z -z_0| \le \delta$. Assume the induction hypothesis: for all $m - n < k$ it follows $||F_{nm} - z||_{|z-z_0|\le\delta,s \in \mathcal{S}} < \rho$ and consequently (by the triangle inequality) $||F_{nm} -z_0||_{|z-z_0|\le\delta,s \in \mathcal{S}} < \delta+\rho < P$ so $F_{nm} \in \mathcal{K}$. We now show this implies the result when $m-n = k$. To save space, and maintain clarity, in the following all the supremum norms $||...||$ are taken across $|z - z_0| \le \delta$ and $s \in \mathcal{S}$. Again, recall that $\phi_n(s, F_{(n+1)m}) = F_{nm}$, we are going to abuse this identity all over again.

\begin{eqnarray*}
||F_{n(n+k)} - z|| &\le& \sum_{j=0}^{k-1} ||F_{(n+j)(n+k)} - F_{(n+j+1)(n+k)}|| + ||F_{(n+k)(n+k)} - z||\\
&=& \sum_{j=0}^{k-1} ||\phi_{n+j}(s, F_{(n+j+1)(n+k)}) - F_{(n+j+1)(n+k)}|| + ||\phi_{n+k} - z||\\
&\le& \sum_{j=0}^{k} \rho_{n+j}\\
&\le& \sum_{j=N+1}^\infty \rho_j < \rho\\
\end{eqnarray*}

Where we've bounded $||\phi_{n+j} (s, F_{(n+j+1)(n+k)}) - F_{(n+j+1)(n+k)}||_{|z-z_0|\le\delta,\mathcal{S}}$ by $||\phi_{n+j}(s, z)- z||_{z \in \mathcal{K},s \in \mathcal{S}} = \rho_{n+j}$ which is justified since $F_{(n+j+1)(n+k)} \in \mathcal{K}$ per the induction hypothesis since $(n + k) - (n + j + 1) = k - j - 1 < k$. Therefore $||F_{n(n+k)} - z|| < \rho$ and $F_{n(n+k)}(s, z) \in \mathcal{K}$ for all $|z - z_0| \le \delta$ and $s \in \mathcal{S}$. By strong induction, this tells us that $|F_{nm}(s, z) - z| < \rho$ for all $m \ge n > N$ while $s \in \mathcal{S}$ and $|z - z_0| \le \delta$. 

About an arbitrary point $z_0 \in \mathbb{C}_{\Re(z)>0}$ there is a compact neighborhood where $|F_{nm}(s, z) - z| < \rho$ for all $s \in \mathcal{S}$. Since $\mathcal{K}$ is a compact disk within $\mathbb{C}_{\Re(z)>0}$, and each compact disk can be covered by a finite number of disks within $\mathbb{C}_{\Re(z)>0}$, we can safely state there exists $N'$ large enough such when $m \ge n > N'$

\[
||\OmSum_{j=n}^m \phi_j(s,z)\bullet z - z||_{\mathcal{K},\mathcal{S}} < \rho
\]

\[
||\phi_n(s,\phi_{n+1}(s,...\phi_m(s,z))) - z||_{\mathcal{K},\mathcal{S}} < \rho
\]
\end{proof}

Now towards our actual function we approach. We are prepared to prove the heart of this paper. Namely that the function

\[
\U(s) = \OmSum_{j=1}^\infty z + e^{(s-j)z} \bullet z
\]

is entire in $s$ and satisfies the functional equation $\U(s + 1) - \U(s) = e^{s\U(s)}$. We have layered our reasoning so far, so this Theorem should be easier to spot by this point. It is essentially the previous two Theorems of the previous two sections, excepting we have an extra variable dangling in there to cause trouble. If this paper were the production of a TV show, this is the series premiere.

\begin{theorem}\label{thm4A}
The function $\U = \OmSum_{j=1}^\infty z + e^{(s-j)z} \bullet z$ is an entire function in $s$ satisfying the functional equation

\[
\Delta \U = \U(s + 1) - \U(s) = e^{s\U(s)}
\]

It is also holomorphic in $z$ for $\Re(z) > 0$.
\end{theorem}

\begin{proof}
Denote $\phi_j (s, z) = z + e^{(s-j)z}$. Choose a set $\mathcal{S} \subset \mathbb{C}$ in which

\[
\sum_{j=1}^\infty ||\phi_j(s,z) - z||_{\mathcal{K},\mathcal{S}} < \infty
\]

for all compact disks $\mathcal{K} \subset \mathbb{C}_{\Re(z)>0}$.\footnote{I'm repeating myself but encore: The set $\mathcal{S}$ need not be bounded, it solely need be a set in which the aforementioned condition is met. For example $\mathcal{S}$ could equal the set $\{s \in \mathbb{C}\, | \,\Re(s) <0,\, |\Im(s)| < 1\}$.}

Choose a point $z_0 \in \mathbb{C}_{\Re(z)>0}$ and pick a compact disk $\mathcal{K} \subset \mathbb{C}_{\Re(z)>0}$ centered at $z_0$ with radius $P$, so that $\mathcal{K} = \{z \in \mathbb{C} \,|\, |z - z_0| \le P\}$ and $\mathcal{K} \subset \mathbb{C}_{\Re(z)>0}$. Choose $\delta > 0$ such that $|z - z_0| < 2\delta < P$, so $z$ still resides in the interior of $\mathcal{K}$. Starting with Lemma \ref{lma4A}, using the same notation $F_{nm} = \OmSum_{j=n}^m \phi_j (s, z) \bullet z = \phi_n(s, \phi_{n+1}(s, ...\phi_m(s, z)))$, we know that $N$ can be chosen large enough so that $||F_{nm}(s, z) - z||_{|z-z_0|\le 2\delta, \mathcal{S}} < \rho$. Choose $N$ large enough such that $2\delta + \rho < P$ so that $F_{nm}(s, z) \in \mathcal{K}$ since $|F_{nm}(s, z)-z_0| < |z -z_0|+|F_{nm}(s, z)-z| < 2\delta +\rho < P$, while $m \ge n \ge N$. The set $\mathcal{K}$ is bounded by $P + |z_0| = A$ and thusly,

\[
||F_{nm}(s, z)||_{|z-z_0|\le2\delta,s \in \mathcal{S}} \le A
\]

Taking the derivatives of $F_{nm}(s, z)$, namely $\frac{d^k}{dz^k}F_{nm}(s,z) = F^{(k)}_{nm}(s,z)$; we can
write something that should be familiar, \`{a} la Cauchy:

\[
F^{(k)}_{nm}(s,z) = \frac{k!}{2\pi i} \int_{|z-z_0| = 2\delta} \frac{F_{nm}(s,\xi)}{(\xi - z)^{k+1}}\,d\xi
\]

As we did in Theorem \ref{thm3A} and Theorem \ref{thm2A}, we're going to shrink the domain of interest from $|z - z_0| \le 2\delta$ to $|z - z_0| \le \delta$, and arrive at the bounds

\[
||F^{(k)}_{nm}(s,z)||_{|z-z_0|\le \delta, \mathcal{S}} \le 2A\frac{k!}{\delta^k}
\]

which follow from the integral representation of $F^{(k)}_{nm}(s,z)$. The reader should freely see the routine of this. Again, $N$ should now be chosen large enough so that $\rho < \delta$. Let us expand $F_{N(m+1)}(s, z) = F_{Nm}(s, \phi_{m+1}(s, z))$ about $F_{Nm}$'s Taylor series about $z$:

\begin{eqnarray*}
F_{N(m+1)}(s, z) - F_{Nm}(s, z) &=& F_{Nm}(s, \phi_{m+1}(s, z)) - F_{Nm}(s, z)\\
&=& \sum_{k=1}^\infty F^{(k)}_{nm}(s,z) \frac{(\phi_{m+1}(s,z) - z)^{k}}{k!}\\
\end{eqnarray*}

This series converges for at least $|z - z_0| \le \delta$ and $s \in \mathcal{S}$ by the bounds $||F^{(k)}_{Nm}(s, z)|| < 2A\frac{k!}{\delta^k}$ and $||\phi_{m+1}(s, z) - z|| < \rho < \delta$. These should be observed as the same inequalities as before. So when we apply the supremum norm we get something which should be familiar by this point in the paper. It's exactly what we've seen before, except there is an extra variable. To save space, each supremum norm in the following is across $s \in \mathcal{S}$ and $z \in \mathcal{K}$ unless stated otherwise.

\begin{eqnarray*}
||F_{N(m+1)} - F_{Nm}||_{|z-z_0|\le\delta} &\le& \sum_{k=1}^\infty ||F^{(k)}_{Nm}||_{|z-z_0|\le\delta} \frac{||\phi_{m+1}(s,z)- z||^k_{|z-z_0|\le\delta}}{k!}\\
&=& \sum_{k=1}^\infty  2 A \frac{k!}{\delta^k}\frac{||\phi_{m+1}(s,z)- z||^k}{k!}\\
&=&2 A ||\phi_{m+1}(s,z)- z|| \sum_{k=1}^\infty   \frac{\rho^{k-1}}{\delta^k}\\
&=& \frac{2A}{\delta(1-\frac{\rho}{\delta})} ||\phi_{m+1}(s,z)- z||\\
&=& M||\phi_{m+1}(s,z)- z||\\
\end{eqnarray*}

Where we let $M = \frac{2A}{\delta(1-\frac{\rho}{\delta})}$. A few of the above steps were kept brief as the procedure is the same as it was in Theorem \ref{thm3A} and Theorem \ref{thm2A}. The sequence $F_{Nm}(s, z)$ can be now seen to be Cauchy in $m$ using the supremum norm across $|z - z_0| \le \delta$ and $s \in \mathcal{S}$. It is given then that $\lim_{m\to\infty} F_{Nm}(s, z) = F_N (s, z)$ is a holomorphic function in $s$ and $z$. Letting $\U(s, z) = \phi_1(s,...\phi_{N-1}(s, F_N (s, z))))$ then $\U$ is a holomorphic function for $s \in \mathcal{S}$ and $|z - z_0| < \delta$. Noting $\mathcal{S}$ was a subset of $\mathbb{C}$ in which

\[
\sum_{j=1}^\infty ||\phi_j(s,z) - z||_{\mathcal{K},\mathcal{S}} < \infty
\]

for any $\mathcal{K}$, and such sets cover $\mathbb{C}$ (e.g. For instance take the cover to be all compact subsets of $\mathbb{C}$); and $z_0$ was an arbitrary point in $\mathbb{C}_{\Re(z)>0}$; it follows $\U(s, z)$ is an entire function in $s$ and holomorphic for $\Re(z) > 0$. All that's left to show is that $\U$ satisfies the desired functional equation. Taking
the sequence of functions

\[
\U_n(s) = \OmSum_{j=1}^n z + e^{(s-j)z}\bullet z = \phi_1(s, \phi_2(s, ...\phi_n(s, z)))
\]

A simple plug and play, as we've tried to detail before, tells us:

\[
\U_n(s) + e^{s\U_n(s)} = \U_{n+1}(s + 1)
\]

so that 

\[
\U(s + 1) = \lim_{n\to\infty} \U_{n+1}(s + 1) = \lim_{n\to\infty} \U_n(s) + e^{s\U_n(s)} = \U(s) +e^{s\U(s)} 
\]

which concludes the proof.
\end{proof}

And voil\`{a}, the function is there. If $f(s, z) = z + e^{sz}$ there is a respective `Gamma' function $\G = \U(s, z)$ holomorphic for $s \in \mathbb{C}$ and $z \in \mathbb{C}_{\Re(z)>0}$. Noting $\U$ depends on both $s$ and $z$, it is perfectly fine to fix $z$ which will give us an entire function in one variable. This will give us a solution $y(s)$ to $\Delta y = e^{sy}$ which is an entire function in $s$.

The value of $z$ serves as a parameter which tells us our initial conditions. Controlling $z$ such that $\U(0) = \xi$ for $\xi$ an arbitrary number is a difficult task though. All the author will say in this paper is that $\U(0)$ depends on $z$. How it
depends on $z$ is left for another day.

With this being shown we can see our equivalent formulation of the Picard-Lindel\"{o}f problem is solved. When $\mathcal{T} y = y(s + 1)$ and $f(s, z) = z + e^{sz}$ then a solution to the equation $\mathcal{T} y = f(s, y)$ is given by the function:

\[
\U(s) = \OmSum_{j=1}^\infty z + e^{(s-j)z} \bullet z
\]

To make it more suggestive, the solution can be written:

\[
\G_f = \OmSum_{j=1}^\infty f(s - j, z) \bullet z
\]\\

The author breathes here to diverge briefly on some results we won't see in detail until Section \ref{sec6}, but which give context to the interesting nature of $\U$. The function $\U(s, z) \to z$ as $\Re(s) \to -\infty$. This is demonstrated by the relation $\U(s - n) = \OmSum_{j=n+1}^\infty z + e^{(s-j)z}\bullet z$. And our intuition should tell us that $\OmSum_{j=n+1}^\infty \phi_j (s, z) \bullet z \to z$ as $n \to \infty$. Just how $\sum_{j=n+1}^\infty a_j \to 0$, or $\prod_{j=n+1}^\infty b_j \to 1$, as $n \to \infty$. This in fact follows from Lemma \ref{lma4A}. This informs an intuition that $e^{s\U(s)} = \U(s + 1) - \U(s) \to 0$ exponentially as $\Re(s) \to -\infty$, because $e^{s\U(s)}$ should look like $e^{sz}$ eventually. Which allows us a self referential identity we get for free. The author refers to this as the telescoping identity:

\begin{eqnarray*}
\sum_{j=1}^\infty e^{(s-j)\U(s-j)} &=& \sum_{j=1}^\infty \U(s - j + 1) - \U(s - j)\\
&=& \lim_{n\to\infty} \U(s) - \U(s - n)\\
&=& \U(s) - z\\
\end{eqnarray*}

This affords us the knowledge $\U(s) - z \to 0$ exponentially as $\Re(s) \to -\infty$. It's natural to wonder what something like the Mellin transform of $\U(-t) - z$ might look like. The author certainly does.

\section{A brief note on the general case}\label{sec5}
\setcounter{equation}{0}

In the above we solved for our troublesome equation $y(s + 1) = \mathcal{T} y = f(s, y)$ when $f(s, z) = z + e^{sz}$. This section is not meant to prove anything, just state generalizations to any $f(s, z)$, these results may or may not be apparent. Their proofs are done exactly as the proofs of Lemma \ref{lma4A} and Theorem \ref{thm4A} were done. As such they will be writ without proof.

Despite this, it would be ignorant of me to not give at least the intuition necessary to visualize the results. To commence, we will attempt to reframe the statement of Lemma \ref{lma4A}. In order to do this we must yet again give clear meaning to our infinite composition. In Lemma \ref{lma4A} we proved our sequence of functions $\phi_j (s, z) = z + e^{(s-j)z}$ when composed together, to form $\OmSum_{j=n}^m \phi_j (s, z)\bullet z$, get arbitrarily close to $z$, in a uniform fashion, as $n \ge m > N$ gets arbitrarily large. The trick which is inherent in reproving this, is that we don't always have nice conditions on the domain and range of $\phi_j (s, z)$. Before, $\phi_j (s, z) = z + e^{(s-j)z}$ was an entire function in both variables for all $j$, but the summability criterion was lost for a large portion of $z$ values.

By the summability criterion, again, the author means

\[
\sum_{j=1}^\infty ||\phi_j(s, z) - z||_{\mathcal{K},\mathcal{S}} < \infty
\]

on specified domains, where $s \in \mathcal{S}$ and $z \in \mathcal{K}$. In the general case the variable $s$ is easy to handle and nothing really changes from what we've already written. For the $z$ variable, it's a bit trickier and the previous case we handled in Section \ref{sec4} can shed light on this. When $\phi_j(s, z) = z + e^{(s-j)z}$ then $z$ can take values in the complex plane but the summability condition can only be satisfied for $\mathcal{K} \subset \mathbb{C}_{\Re(z)>0}$. Equally as troublesome $\phi_j (s, \cdot) : \mathbb{C}_{\Re(z)>0} \to \mathbb{C}_{\Re(z)>0}$. However, magically, in the final result we managed to prove $\OmSum_{j=1}^\infty \phi_j(s,z)$ is holomorphic for $z \in \mathbb{C}_{\Re(z)>0}$.

The boiled down result for any $\phi_j(s,z)$ is $\OmSum_{j=1}^\infty \phi_j(s,z)$ converges everywhere $\phi_j$ satisfies the summability criterion--so long as these expressions are meaningful. Our theorem should truly read: if $F_n = \OmSum_{j=1}^n \phi_j (s, z) \bullet z$ is a meaningful thing, and $\sum_{j=1}^\infty ||\phi_j (s, z) - z|| < \infty$, then $\lim_{n\to\infty} F_n = F$ is a holomorphic function. Sadly, a discussion of domains is necessary.

To elucidate the situation: let $\mathcal{B} \subseteq \mathcal{G} \subseteq \mathbb{C}$ be open sets. Suppose the sequence of holomorphic functions $\phi_j (s, z) : \mathcal{S} \times \mathcal{G} \to \mathcal{G}$ satisfy the summability criterion for all compact disks $\mathcal{K}$ within $\mathcal{B}$:

\[
\sum_{j=1}^\infty ||\phi_j(s,z) - z||_{\mathcal{K},\mathcal{S}} < \infty
\]

Then $||\OmSum_{j=n}^m \phi_j (s, z) \bullet z - z||_{\mathcal{K},\mathcal{S}}$ can be made arbitrarily small depending on how large we let $N$ get, where $m \ge n > N$. Looking forward, this tells us $\Phi(s, z) = \OmSum_{j=1}^\infty \phi_j (s, z) \bullet z$ is a holomorphic function for $s \in \mathcal{S}$ and $z \in \mathcal{B}$, where its range is $\mathcal{G}$. So that $\Phi(s, z) : \mathcal{S} \times \mathcal{B} \to \mathcal{G}$. In the last section $\phi_j (s, z) = z + e^{(s-j)z}$, and $\mathcal{S}$ could be an arbitrary subset of the complex plane (wherein the summability criterion was met), $\mathcal{B} = \mathbb{C}_{\Re(z)>0}$, and $\mathcal{G} = \mathbb{C}$--so that $\U(s, z) : \mathcal{S} \times \mathbb{C}_{\Re(z)>0} \to \mathbb{C}$. Where we could luckily cover $\mathbb{C}$ with a collection of sets $\{ \mathcal{S}_l \}_{l=1}^\infty$, so in the end it could be deduced $\U(s, z) : \mathbb{C} \times \mathbb{C}_{\Re(z)>0} \to \mathbb{C}$.

We can state the generalizations of both Lemma \ref{lma4A} and Theorem \ref{thm4A}. The proofs of which should be made appropriately by looking at the mentioned Lemma and Theorem.

\begin{lemma}
Let $\mathcal{S} \subset \mathbb{C}$ and $\mathcal{G} \subseteq \mathbb{C}$ where $\mathcal{G}$ is open. Suppose $\{\phi_j (s, z)\}_{j=1}^\infty$ is a sequence of holomorphic functions for $s \in \mathcal{S}$ and $z \in \mathcal{G}$. Suppose also, for $s$ fixed, $\phi_j (z) : \mathcal{G} \to \mathcal{G}$. If for all compact disks $\mathcal{K}$ where $\mathcal{K} \subset \mathcal{B} \subseteq \mathcal{G}$, where $\mathcal{B}$ is open:

\[
\sum_{j=1}^\infty ||\phi_j(s, z) - z||_{\mathcal{K},\mathcal{S}} < \infty
\]

then for all $\rho > 0$ there exists $N$ such when $m \ge n > N$:

\[
||\OmSum_{j=n}^m \phi_j (s, z) \bullet z - z||_{\mathcal{K},\mathcal{S}} < \rho
\]
\end{lemma}

Getting from this Lemma to the following Theorem follows as the above three theorems in the previous three sections have gone.

\begin{theorem} 
Let $\mathcal{S} \subset \mathbb{C}$ and $\mathcal{G} \subseteq \mathbb{C}$ where $\mathcal{G}$ is open. Suppose $\{\phi_j (s, z)\}_{j=1}^\infty$ is a sequence of holomorphic functions for $s \in \mathcal{S}$ and $z \in \mathcal{G}$. Suppose also, for $s$ fixed, $\phi_j (z) : \mathcal{G} \to \mathcal{G}$. If for all compact disks $\mathcal{K}$ where $\mathcal{K} \subset \mathcal{B} \subseteq \mathcal{G}$, where $\mathcal{B}$ is open:

\[
\sum_{j=1}^\infty ||\phi_j(s, z) - z||_{\mathcal{K},\mathcal{S}} < \infty
\]

then the sequence of functions $\Phi_n(s, z) : \mathcal{S} \times \mathcal{B} \to \mathcal{G}$

\[
\Phi_n(s, z) = \OmSum_{j=1}^n \phi_j (s, z) \bullet z
\]

converge uniformly to $\Phi(s, z)$ for $s \in \mathcal{S}$ and on compact subsets of $\mathcal{B}$.
\end{theorem}

From these two theorems the general case of our initial problem is much easier to solve. The strange Picard-Lindel\"{o}f problem is patchwork solved. The discussion of domain and range must be made explicit in each case though. The essential theorem must be stated loosely and without rigor to be absorbed. If for all $\mathcal{K}$, a compact disk within $\mathcal{B} \subseteq \mathcal{G}$, and $f(s, \cdot) : \mathcal{G} \to \mathcal{G}$ where a summability criterion is met

\[
\sum_{j=1}^\infty ||f(s - j, z) - z||_{\mathcal{K},\mathcal{S}} < \infty
\]

then $y(s+ 1) = \mathcal{T} y = f(s, y(s))$ exists and is holomorphic for $s \in \mathcal{S}$ and can be constructed by $y(s, z) =\OmSum_{j=1}^\infty f(s-j, z)\bullet z$ for $z \in \mathcal{B}$ which serves as a parameter. This is the sussed out meaning of our solution of the modified Picard-Lindel\"{o}f problem.

When we make a change of operator and concern ourselves with $y(s+1)-y(s) = \Delta y = q(s, y)$, then our condition appears a bit cleaner; in it $q(s, z)$ must satisfy:

\[
\sum_{j=1}^\infty||q(s - j, z)||_{\mathcal{K},\mathcal{S}} < \infty
\]

Or rather, that the sum:

\[
\sum_{j=1}^\infty q(s - j, z)
\]

converges compactly normally and $\OmSum_{j=1}^n z + q(s-j,z) \bullet z$ is a meaningful thing. As soon as we let $q$ or $f$ be entire functions in $z$ then the problem is much easier to discuss. However, the author felt it necessary to frame the language to be as general as possible.

For those of you familiar with Ramanujan summation; this is a condition, if not the same, uniform with Ramanujan's condition. When $q(s, z) = q(s)$ and $q$ is constant in $z$, this entire paper boils into a very similar discussion. Where instead we are trying to solve for $\Delta y = q(s)$ and $y$ is referred to as the indefinite sum of $q$. The beauty of Ramanujan Summation is that Ramanujan managed to enlarge the domain in which this condition works. The author can patch work loosen these conditions, but he imagines if Ramanujan lived longer he would've put the author to shame.\\

In an attempt to satiate the reader, the author will give the solutions to $\G_f$ for the functions defined in Equations \ref{eq:1C} and \ref{eq:1D}. Sometimes dealing with actual functions makes the idea clearer. If

\[
f(s, z) = z + \frac{z^2}{s^2} + \frac{z^3}{z^3}
\]

then the appropriate Gamma function $\G_f = \mathcal{Q}$ which satisfies $f(s, \mathcal{Q}(s)) = \mathcal{Q}(s + 1)$ is given by the equation:

\[
\mathcal{Q} = \OmSum_{j=1}^\infty z + \frac{z^2}{(s-j)^2} + \frac{z^3}{(s-j)^3} \bullet z
\]

This $\mathcal{Q}(s, z)$ is holomorphic in $s$ for $s \in \mathbb{C}/\mathbb{N}$ and entire in $z$. It also satisfies the very weird equation:

\[
\mathcal{Q}(s+1) = \mathcal{Q} + \frac{\mathcal{Q}^2}{s^2} + \frac{\mathcal{Q}^3}{s^3}
\]

\[
\Delta \mathcal{Q} = \frac{\mathcal{Q}^2}{s^2} + \frac{\mathcal{Q}^3}{s^3}
\]

The $z$-value here, again serves as a parameter. It determines our initial condition of $\mathcal{Q}(s)$ at $\mathcal{Q}(0)$. Since $\mathcal{Q}$ is entire $z$ we can conclude there is possibly only one value $\zeta$ in which $\mathcal{Q}(0, z) \neq \zeta$. From this, for every other value $\xi \neq \zeta$ there is a $\mathcal{Q}(s)$ in which $\mathcal{Q}(0) = \xi$. Which is neat.

If our problem is a bit kookier and $f(s, z) = z\cos(2^sz) + \frac{\sin(z)}{s^2 +1} + \frac{\sin^2(z)}{(s^2+1)^2}$ then $\G_f = \mathcal{P}$ is given by:

\[
\mathcal{P} = \OmSum_{j=1}^\infty z\cos(2^{s-j}z) + \frac{\sin(z)}{(s-j)^2 +1} + \frac{\sin^2(z)}{((s-j)^2+1)^2}\bullet z
\]

It is also entire in $z$ and holomorphic for all $s \in \mathbb{C}$ where $s \neq j \pm i$ for $j \in \mathbb{N}$. It satisfies the excessive equation:

\[
\mathcal{P}(s+1)= \mathcal{P}\cos(2^s\mathcal{P}) + \frac{\sin(\mathcal{P})}{s^2 +1} + \frac{\sin^2(\mathcal{P})}{(s^2+1)^2}
\]

We leave a few exercises for the reader: Prove $\mathcal{Q}$ is a meromorphic function in $s$. Prove $\mathcal{P}$ is not a meromorphic function in $s$. Prove the order of the poles of $\mathcal{Q}$ at $s = j$ are $3^j$. Prove each singularity of $\mathcal{P}$ is an essential singularity excepting the second order poles at $\pm i$.

\section{Asymptotics of $\U$ on a half-plane}\label{sec6}
\setcounter{equation}{0}

The solutions of $y = \U(s)$ where $\Delta \U = e^{s \U(s)}$ are mysterious so far. Despite the nature of our solution being an infinite composition of functions, there is still a platform for working with them. Here is where we begin to carve away its shape at infinity. This is going to be tricky. One of the most magnificent results concerning the $\G$-function is its behaviour at infinity. Stirling's asymptotic formula is invaluable and a sheer work of art. It can be stated bluntly as:

\[
\G(s) \sim \sqrt{2\pi} s^{s+1/2}e^{-s}\, \text{as}\,|s| \to \infty \,\text{while}\, |\arg(s)| < \pi
\]

A beautiful proof of this result resides in Reinhold's \emph{Classical Topics In Complex Function Theory} \cite{ref14}, where there is a detailed discussion of the error term as well. The goal in this section is to do something similar with $\U$.

As an entire function, how does $\U$ behave as we increase $|s|$? I briefly mentioned above that as $\Re(s) \to -\infty$ our solution $\U(s)$ decays exponentially to the parameter $z$. I've said nothing yet about the behaviour as $\Im(s)$ approaches infinity. This section is here to set the aforementioned heuristics in stone, and elaborate by discussing $\U$'s behaviour as $\Im(s)$ grows.

As a brief digression, the author failed to find any substantial information on how $\U$ behaves as $\Re(s) \to \infty$ or as $|s| \to \infty$ when $|\arg(sz)| \le \pi/2$. The only thing he could wrap his head around is that it must grow fast, and very fast, somewhere in that domain. He rudely conjectures that $M(r) = \sup_{|s|\le r} |\U(s)|$ grows faster than any super-exponential. Meaning for all $a > 1$:

\[
\lim_{r\to\infty} \frac{a^{a^{...(n\,times)...^a}}}{M(r)} = 0
\]

Where $n = [r] = \text{greatest natural number less than or equal to r}$. He sadly failed to prove this fact for arbitrary $\Re(z) > 0$ despite numerous attempts. When $z \in \mathbb{R}^+$ the result is to say the least obvious; as $\U(x) > 0$ and $\U(x+1) > e^{x\U(x)}$.

What we aim to prove in this section is:

\[
\U(s) - z \sim \frac{e^{(s-1)z}}{1 - e^{-z}} \, \text{as}\, |s| \to \infty \, \text{while}\, \pi/2 < | \arg(sz)| \le \pi
\]

The essential statement, in English is: wherever $\frac{e^{(s-1)z}}{1 - e^{-z}}$ tends to zero, $\U(s) - z$ tends to zero also, in an asymptotic fashion. To prove this result we will progress by showing that $\U(s+ 1)-\U(s) \sim e^{sz}$ in this domain, and then show that $\U(s)-z \sim \frac{e^{(s-1)z}}{1 - e^{-z}}$ using the telescoping identity, briefly mentioned at the end of Section \ref{sec4}.

The telescoping identity is the simple manipulation of symbols:

\begin{eqnarray*}
\U(s) - z &=& \U(s) - \lim_{n\to\infty}\U(s - n)\\
&=& \lim_{n\to\infty} \sum_{j=1}^n \U(s - j + 1) - \U(s - j)\\
&=& \sum_{j=1}^\infty \U(s - j + 1) - \U(s - j)\\
&=& \sum_{j=1}^\infty e^{(s-j)\U(s-j)}\\
\end{eqnarray*}

A lemma is needed to complete this proof, this lemma acts more generally than we use it. However, we'll frame it in a localized setting and worry only about $\U(s)$. The take away from it is, if

\[
G(s, z) =\OmSum_{j=1}^\infty z + \ell_j (s, z) \bullet z
\]

and $\ell_j (s, z)$ tends to zero fast enough as $s \to A$ then $G(s, z) \to z$ as $s \to A$. It is a kind of `pulling the limit through the composition' type theorem. Wherein exchanging the infinite composition with the limit is made rigorous. Our phrasing will be chosen solely for the scenario of $\U(s)$ though.\footnote{This is where our summability criterion shines, and why we allowed unbounded domains when talking of it.}

\begin{lemma}\label{lma6A}
The function $\U(s) \to z$ as $|s| \to \infty$ when $\pi/2 < | \arg(sz)| \le \pi$.
\end{lemma}

\begin{proof}
The essential method of this proof is to consider

\[
\U_n(s, z) =\OmSum_{j=1}^n z + e^{(s-j)z} \bullet z
\]

And show that as $|s| \to \infty$ in the appropriate half-plane $\U_n(s) \to z$. Using induction, the function $\U_1(s, z) = z + e^{(s-1)z} \to z$ as $|s| \to \infty$ in the sector $\pi/2 < \kappa < | \arg(sz)| \le \pi$. The convergence is also uniform in neighborhoods of $z$.

Assume that the same follows for $\U_n(s, z)$. Well then, $\U_{n+1}(s, z) = \U_n(s, z +e^{(s-n-1)z})$ and the function $z + e^{(s-n-1)z} \to z$ uniformly in $z$ as $|s| \to \infty$ in the appropriate sector $\pi/2 < \kappa < \arg(sz) \le \pi$. Therefore $\U_{n+1}(s, z) \to z$ uniformly.

By Lemma \ref{lma4A}, and our summability criterion, for all $\rho > 0$ there exists $N$ such $|\OmSum_{j=n}^m z+e^{(s-j)z} \bullet z-z| < \rho$ for all $z \in \mathcal{K}$ while $\pi/2 < \kappa \le |\arg(sz)| \le \pi$ and $m \ge n > N$. Call $w \in \mathcal{N}$ the neighborhood $|w - z| < \rho$, then since $\U_{N-1}(s, z) \to z$ uniformly for $z \in \mathcal{N}$ as $|s| \to \infty$ while $\pi/2 < \kappa \le|\arg(sz)| \le \pi$ ; we must have $\U(s, z) \to z$. Since the tail of the composition remains bounded in $s$ and also approaches $z$ as we get further and further out, the beginning of the composition forces convergence to $z$ as we let $s$ increase in the half-plane $\pi/2 < | \arg(sz)| \le \pi$.
\end{proof}

Finally, some simple satiating asymptotics for $\U(s)$.

\begin{theorem}\label{thm6A}
The function $\U(s + 1) - \U(s) \sim e^{sz}$ as $|s| \to \infty$ when $\pi/2 <|\arg(sz)| \le \pi$. The function $\U(s) - z \sim \frac{e^{(s-1)z}}{1-e^{-z}}$.
\end{theorem}

\begin{proof}
To show $\U(s + 1) - \U(s) \sim e^{sz}$ we need only show $s(\U(s) - z) \to 0$ due to the fact

\[
\frac{\U(s + 1) - \U(s)}{e^{sz}}= e^{s(\U(s)-z)}
\]

The essential trick and the crux of the proof relies on showing for $t \in \mathbb{R}^+$ and large, that $|\U(1 + \alpha t) - \U(\alpha t)| = |e^{\alpha t \U(\alpha t)}| \le Ce^{-t\rho}$ as $t$ gets arbitrarily close to $\infty$ for some arbitrary $\rho > 0$ and $|\alpha| = 1$ with $\pi/2 < \arg(\alpha z) \le \pi$. The heavy lifting has been handled in Lemma \ref{lma6A}. Fix $\kappa$ where $\pi/2 < \kappa < \pi$.

There exists $R$ large enough so that $-\epsilon < \arg(\U(s)) - \arg(z) < \epsilon$ and $\pi/2 + \epsilon < \kappa$. Choose $\alpha$ such that $|\arg(\alpha z)| > \kappa$. Then $\kappa < |\arg(\alpha z)| < |\arg(\alpha \U(\alpha t))| +|\arg z - \arg \U(\alpha t)|$ and we can intuit $|\arg(\alpha \U (t\alpha))| > \kappa - \epsilon > \pi/2$. Therefore $\alpha \U(\alpha t)$ resides in the left half plane. Therefore $e^{\alpha t \U(\alpha t)}$ tends to zero as $t \to \infty$. It inherently tends to zero exponentially. Therefore $\U(s + 1) - \U(s)$ tends to zero exponentially as $|s| \to \infty$ in the appropriate half-plane. We can explicitly bound $|\U(s + 1) - \U(s)|$ by $|e^{\beta s}|$ for some $|\beta| > 0$ and $\arg(\beta) = \arg(z)$ in the sector $\pi/2 < \kappa \le |\arg(sz)| \le \pi$.

Using the telescoping identity:

\begin{eqnarray*}
|\U(s) - z| &\le& \sum_{j=1}^\infty |\U(s - j + 1) - \U(s - j)|\\
&\le& \sum_{j=1}^\infty |e^{\beta(s-j)}|\\
&\le&\frac{e^{\Re (\beta s)- \Re(\beta)}}{1 - e^{\Re(\beta)}}\\
\end{eqnarray*}

Which implies $s(\U(s) - z) \to 0$ as $|s| \to \infty$ when $\pi/2 < \arg(sz) \le \pi$, because $\U(s) - z$ must have exponential decay. To refresh why this gives the result, since $s(\U(s) - z) \to 0$:

\[
\frac{\U(s + 1) - \U(s)}{e^{sz}} = e^{s (\U(s) - z)} \to 1 \,\text{as} \,|s| \to \infty \,\text{while}\, \pi/2 < |\arg(sz)| \le \pi
\]

Since $\U(s + 1) - \U(s) \sim e^{sz}$, the telescoping identity tells us that $\U(s) - z \sim \frac{e^{(s-1)z}}{1-e^{-z}}$. We leave this to the reader.
\end{proof}

\section{Closing Remarks}\label{sec7}
\setcounter{equation}{0}

To close our discussion of $\Delta y = e^{sy}$ it is apt to discuss the uniqueness of our solution. And that there is none. Our solutions aren't the only solutions. The author can construct a couple few very simply. The functional equation $\Delta y = e^{sy}$ really only makes sense on the entire complex plane, or a horizontal strip of it which marches off in either direction towards infinity, so we'll stick with solutions of such a form. 

With that being said. Consider the strange function $g(s) = \U(s, 2 +\sin 2\pi s)$, for $|\Im(s)| < \delta$ where $\delta$ is small enough so $\Re(2 + \sin 2\pi s) > 0$. The function $g$ is holomorphic in $s$ since $\U(s, z)$ is holomorphic in $z$ for $\Re(z) > 0$ and entire in $s$. Taking its difference:

\begin{eqnarray*}
\Delta g = g(s + 1) - g(s) &=& \U(s + 1, 2 + \sin 2\pi(s + 1)) - \U(s, 2 + \sin 2\pi s)\\
&=& \U(s + 1, 2 + \sin 2\pi s) - \U(s, 2 + \sin 2\pi s)\\
&=& e^{s \U(s,2+sin 2\pi s)}\\
&=& e^{sg(s)}\\
\end{eqnarray*}

Therein $g$ is also a solution to $\Delta y = e^{sy}$. Generally if $L(s)$ is a $1$-periodic function on some strip of $\mathbb{C}$ where $\Re(L) > 0$, then $\ell(s) = \U(s, L(s))$ is a solution to our first order difference equation. To say the least, this can be rather frustrating when searching for a clean uniqueness theorem.

The best the author can really say at this point in time, is that if our solution $y(s) \to z$ as $\Re(s) \to -\infty$ then $y(s) = \U(s, z)$. Proving this is next to trivial. The expression $y(s) = \OmSum_{j=1}^n z + e^{(s-j)z} \bullet z \bullet y(s-n) = \U_n(s,y(s-n))$ converges to $\U(s,z)$.

Sadly, that's really the best the author could get at the moment. So, although we've solved the modified Picard-Lindel\"{o}f equation, we've only really gotten half way. We have existence, but we do not have uniqueness based on nice conditions. The only really added benefit one gets in the Difference Equation setting is that we have a construction method.

So to summarize: Picard-Lindel\"{o}f gives us the uniqueness and existence of a First Order Differential Equation dependent on some initial condition; the author has given an existence and construction method of a First Order Difference Equation--and said next to nothing about the initial condition/uniqueness problem.\\

It would be interesting to ask a few questions before we close our paper. The author was taught this is good form. Must all solutions of $\Delta y = e^{sy}$ be of the form $y(s) = \U(s, L(s))$ for $L(s)$ a $1$-periodic function? Noticeably when $y = \U(s, L(s))$, $y(s)$ remains bounded as $\Re(s) \to -\infty$. The natural question I like to ask is if there is a solution $y$, such that $y(s)$ is unbounded as $\Re(s) \to -\infty$?\\

We have spent a load of time dealing with the operator $\mathcal{T} y = y(s + 1)$. What if we dealt with a different operator? Maybe something along the lines of $\mathcal{L}y = y(\sqrt{s})$? In such a circumstance could we construct a function that solves an even weirder Picard-Lindel\"{o}f equation $y(\sqrt{s}) = \mathcal{L}y = k(s, y)$ for some function $k$? In fact, so long as:

\[
\sum_{j=1}^\infty k(s^{2^j},z)
\]

is compactly normally convergent, then:

\[
y(s) = \OmSum_{j=1}^\infty k(s^{2^j},z)\bullet z
\]

solves the equation $\mathcal{L}y = k(s, y)$. For instance if $k(s, z) = z + sz^2$ and we let $z, s \in \mathbb{C}$ with $|s| < 1$, then the satisfying equation would be:

\[
y(\sqrt{s}) = y(s) + sy(s)^2
\]

Oddly enough,

\[
y= \OmSum_{j=1}^\infty z + s^{2^j}z^2\bullet z
\]

works perfectly fine. Can you see why?\\

Lastly, we ask a more detailed question. Can the actual Picard-Lindel\"{o}f problem be recovered from our solution? That is to say, can we use our knowledge of $\mathcal{T} y = y(s + 1) = f(s, y)$ to recover solutions to $y' = g(s, y)$. The answer to which can be intuited using a now blunt tool. The author calls it a compositional integral. The below requires a strong appreciation of our notation to make sense. Please forgive the author if it appears terse.

Rather than working with $\Delta y = y(s+ 1)-y(s)$, we can work with the operator $\Delta_h y = y(s + h) - y(s)$. Switching to this scenario is not very difficult. Whereas earlier we took infinite compositions of the form $f(s - 1, f(s - 2, f(s - 3, ...)))$ we now take infinite compositions in the form $\eta(s) = f(s-h, f(s-2h, f(s-3h, ...)))$. Now instead, $\eta$ satisfies $\eta(s + h) = f(s, \eta(s))$. Letting $h \to 0$ makes a kind of continuous sweep of the functions... and they are composed together. Whatever really that means. From this identification we can play around a bit.

Suppose we wanted to solve the equation $y' = e^{sy}$. The compositional integral can save us here. We will take a note from Leibniz in our notation here. Let $ds$ exist in the right half plane and let it approach $0$. It could perfectly well be considered an infinitesimal. Let $f(s, z) = z + e^{sz}ds$, which is almost as before.

\[
\mu(s) = f(s - ds, f(s - 2ds, f(s - 3ds, ...))) = \OmSum_{j=1}^\infty z + e^{(s-jds)z}ds \bullet z
\]

Satisfies the difference equation:

\[
\mu(s + ds) - \mu(s) = e^{s\mu(s)}ds
\]

Which if we let our $ds \to 0$ and somehow concoct a proof that $\mu$ exists, we arrive at the differential equation:

\[
\mu' = e^{s\mu}
\]

This can be written in a succinct Riemann-like formula which is equivalent to the above heuristics:

\[
\mu(s) = \lim_{n\to\infty} \OmSum_{j=1}^\infty z + \frac{1}{n}e^{(s-\frac{j}{n})z}\bullet z
\]

Where again $z$ is a parameter and determines the initial condition of our problem. The reader may freely see the similarity to the integral. In fact if we switch from $e^{sz}$ to an arbitrary $g(s)$ which is constant in $z$, we arrive at a familiar Riemann Sum:

\begin{eqnarray*}
G(s) &=& \lim_{n\to\infty} \OmSum_{j=1}^\infty z + \frac{1}{n}g(s-\frac{j}{n}) \bullet z\\
&=& z + \sum_{j=1}^\infty \frac{1}{n}g(s-\frac{j}{n})\\
&=&z + \int_{-\infty}^s g(t)\,dt\\
\end{eqnarray*}

The author will not expand further on this at the moment, as it's all up in the air at this point, excepting a single proposed clean notation for the compositional integral:

\[
\mu(s) = \int_{-\infty}^s e^{wz} dw \bullet z
\]

which again satisfies $\mu' = e^{s\mu}$ and in this instance $\mu(-\infty) = z$.\\

Sadly this paper only scratches the surface. Attacking problems like the above require far more nuance and care than the author has. It is not that each solution is necessarily hard, it is that the author lacks the vocabulary to phrase said ideas in a single general framework. So I ask: Does the reader?\\

\end{document}